\documentclass[a4paper,11pt]{article}
\usepackage[latin1]{inputenc}
\usepackage[english]{babel}
\usepackage{amsmath}
\usepackage{amsfonts}
\usepackage{amssymb}
\usepackage{epsfig}
\usepackage{amsopn}
\usepackage{amsthm}
\usepackage{color}
\usepackage{graphicx}
\usepackage{enumerate}
\newtheorem{theorem}{Theorem}[section]

\newtheorem{lemma}[theorem]{Lemma}
\newtheorem{proposition}[theorem]{Proposition}
\newtheorem{definition}[theorem]{Definition}
\newtheorem{remark}[theorem]{Remark}
\newtheorem*{theorem*}{Theorem}
\newtheorem*{lemma*}{Lemma}
\newtheorem*{remark*}{Remark}
\newtheorem*{definition*}{Definition}
\newtheorem*{proposition*}{Proposition}
\newtheorem*{corollary*}{Corollary}
\numberwithin{equation}{section}
\newcommand{\real}{\mathbb{R}}

\let\ced=\c         

\def\e{\varepsilon}        

\def\qed{\,\unskip\kern 6pt \penalty 500
\raise -2pt\hbox{\vrule \vbox to8pt{\hrule width 6pt
\vfill\hrule}\vrule}\par}
\definecolor{darkblue}{rgb}{0.05, .05, .65}
\definecolor{darkgreen}{rgb}{0.1, .65, .1}
\definecolor{darkred}{rgb}{0.8,0,0}
\newcommand{\beqn}{\begin{equation}}
\newcommand{\eeqn}{\end{equation}}
\newcommand{\bear}{\begin{eqnarray}}
\newcommand{\eear}{\end{eqnarray}}
\newcommand{\bean}{\begin{eqnarray*}}
\newcommand{\eean}{\end{eqnarray*}}
\begin{document}

\title{\huge \bf Blow up profiles for a quasilinear reaction-diffusion equation with weighted reaction with linear growth}

\author{
\Large Razvan Gabriel Iagar\,\footnote{Instituto de Ciencias
Matem\'aticas (ICMAT), Nicol\'as Cabrera 13-15, Campus de
Cantoblanco, 28049, Madrid, Spain, \textit{e-mail:}
razvan.iagar@icmat.es}, \footnote{Institute of Mathematics of the
Romanian Academy, P.O. Box 1-764, RO-014700, Bucharest, Romania.}
\\[4pt] \Large Ariel S\'{a}nchez,\footnote{Departamento de Matem\'{a}tica
Aplicada, Ciencia e Ingenieria de los Materiales y Tecnologia
Electr\'onica, Universidad Rey Juan Carlos, M\'{o}stoles,
28933, Madrid, Spain, \textit{e-mail:} ariel.sanchez@urjc.es}\\
[4pt] }
\date{}
\maketitle

\begin{abstract}
We study the blow up profiles associated to the following second order reaction-diffusion equation with non-homogeneous reaction:
$$
\partial_tu=\partial_{xx}(u^m) + |x|^{\sigma}u,
$$
with $\sigma>0$. Through this study, we show that the non-homogeneous coefficient $|x|^{\sigma}$ has a strong influence on the blow up behavior of the solutions. First of all, it follows that finite time blow up occurs for self-similar solutions $u$, a feature that does not appear in the well known autonomous case $\sigma=0$. Moreover, we show that there are three different types of blow up self-similar profiles, depending on whether the exponent $\sigma$ is closer to zero or not. We also find an explicit blow up profile. The results show in particular that \emph{global blow up} occurs when $\sigma>0$ is sufficiently small, while for $\sigma>0$ sufficiently large blow up \emph{occurs only at infinity}, and we give prototypes of these phenomena in form of self-similar solutions with precise behavior. This work is a part of a larger program of understanding the influence of non-homogeneous weights on the blow up sets and rates.
\end{abstract}

\


\smallskip


\section{Introduction}

In this paper, we deal with the phenomenon of blow up in finite time for the following reaction-diffusion equation with a non-homogeneous reaction term:
\begin{equation}\label{eq1}
u_t=(u^m)_{xx}+|x|^{\sigma}u, \qquad u=u(x,t), \qquad (x,t)\in\real\times(0,T),
\end{equation}
for $m>1$ and any exponents $\sigma>0$, where the subscript notation in \eqref{eq1} indicates partial derivative with respect to the time or space variable. As usual, we say that a solution $u$ to Eq. \eqref{eq1} blows up in a finite time $T\in(0,\infty)$ if $u(t)\in L^{\infty}(\real)$ for any $t\in[0,T)$, but $u(T)\not\in L^{\infty}(\real)$. The time $T\in(0,\infty)$ as above is called the blow up time of the solution $u$. As a notation, here and throughout the paper by $u(t)$ we will understand the mapping $\real\ni x\mapsto u(x,t)\in\real$, for a given time $t>0$.

A first very interesting aspect of this problem is that finite time blow up in Eq. \eqref{eq1} is completely driven by the weight $|x|^{\sigma}$. Indeed, when the weight does not exist (that is $\sigma=0$), we remain with the following homogeneous equation
\begin{equation}\label{eq1.hom}
u_t=(u^m)_{xx}+u,
\end{equation}
for which it is well-known that typical solutions do not blow up in finite time. More precisely, solution to Eq. \eqref{eq1.hom} with bounded initial condition $u_0$ present exponential growth as $t\to\infty$, but they remain bounded for any time $t>0$, as it can be easily seen by making the change of variable
$$
u(x,t)=e^{t}v(x,t), \quad t=\frac{\log(m-1)\tau}{m-1},
$$
and noticing that the equation satisfied by the new function $v(x,\tau)$ is of porous medium type. It thus follows that blow up in finite time for solutions to Eq. \eqref{eq1} produces only due to the reinforced reaction term with an unbounded weight. Moreover, as it will be seen at the level of self-similar solutions, the condition $m>1$ is also essential for blow up in finite time. If we consider the reaction-diffusion with weighted reaction in its most general form
\begin{equation}\label{eq1.gen}
u_t=(u^m)_{xx}+|x|^{\sigma}u^p,
\end{equation}
with generic exponents $m>0$, $\sigma>0$ and $p>0$, the condition for finite time blow up to occur (at least at the level of self-similar solutions) is
\begin{equation}\label{cond.blow}
K:=\sigma(m-1)+2(p-1)>0,
\end{equation}
see our companion paper \cite{IS2}. We thus devote this work to the study of the very interesting limit case $p=1$ with $m>1$ and $\sigma>0$, while the more general case where $1<p<m$ in \eqref{eq1.gen} is investigated in the above mentioned companion paper \cite{IS2}. This splitting is justified by the fact that the proofs of the results in the case $p=1$ differ from the ones used for $p>1$; in particular, the proofs in the present work have a more geometrical basis.

\medskip

\noindent \textbf{Precedents.} Equations of the form
\begin{equation}\label{eq1.gen.gen}
u_t=\Delta u^m + a(x)u^p,
\end{equation}
with different cases of weight function $a(x)$, appeared in recent literature, in particular for the case when $a(x)$ is a compactly supported function (for example, a characteristic function of a bounded set). The work by Ferreira, de Pablo and V\'azquez \cite{FdPV06} deals with the one-dimensional case. Their results were then generalized to the $N$-dimensional case in \cite{KWZ11, Liang12} and also to the fast diffusion case $m<1$ \cite{BZZ11}. In all these relatively recent papers, interesting properties related to the Fujita-type exponent are established. One surprising fact is that there are important differences concerning the value of the Fujita-type critical exponent, which depends strongly on the space dimension $N=1$, $N=2$ and $N\geq3$ and also in all these cases it is different from the standard exponent of the homogeneous case (with $a(x)\equiv1$). Moreover, in \cite{FdPV06}, blow up rates, sets and profiles are also established for the one-dimensional case. All these works rely deeply on the fact that the non-homogeneous reaction is compactly supported which in fact means that there is no reaction at all close to the space infinity.

Much less is known for equations of the type \eqref{eq1.gen.gen} when $a(x)$ is not compactly supported and, more important, an unbounded function (typically a positive power $a(x)=|x|^{\sigma}$). There are available a number of works establishing the Fujita-type exponent $p_*=m+(\sigma+2)/N$ which satisfies the property that, if when $1<p<p_*$, all the solutions starting from a bounded initial condition $u_0\in L^{\infty}$ blow up in finite time, mostly for the semilinear case $m=1$ \cite{BK87, BL89, Pi97, Pi98, Su02}. In the latter of these \cite{Su02}, the case $m>1$ is considered too, but with the restriction $p>m$, and some conditions on the form and space tail as $|x|\to\infty$ of the initial condition $u_0$ are given in order to insure finite time blow up even for $p>p_*$. More recently, there are also works dealing with finer properties concerning finite time blow up for solutions to such equations, such as blow up rates or sets (see for example the series of papers \cite{GLS, GS11, GLS13}). An interesting problem derived from the presence of a power weight $a(x)=|x|^{\sigma}$ is whether $x=0$ can be a blow up point or not (as pointwisely there no reaction occurs at the origin). In some partial ranges for $p$ and $\sigma$ in \eqref{eq1.gen}, blow up rates were established by Andreucci and Tedeev \cite{AT05}, who also consider different, more complicated nonlinearities, and some extensions to coupled systems of reaction-diffusion equations with weighted reaction were established in \cite{IU08} but only for semilinear equations (that is, when $m=1$).

The goal of the ongoing project to which the present work and the companion work \cite{IS2} belong to, is to understand better how finite time blow up occurs when an unbounded weight appears in the reaction, and in particular establish blow up sets, rates and profiles, which give a rather accurate description of the shape of the solutions when they approach their blow up time. We are now ready to state the main results of this paper.

\bigskip

\noindent \textbf{Main results}. Along this paper, our aim is to classify self-similar solutions to \eqref{eq1}, which are, as explained above, prototypes for the general qualitative properties of solutions; in particular, these special solutions are usually the \emph{blow-up profiles} (that is, patterns to which a solution approached close to its blow-up time) for general solutions, a fact that will be rigorously shown in forthcoming papers. More precisely, let us set
\begin{equation}\label{SSB}
u(x,t)=(T-t)^{-\alpha}f(\xi), \quad \xi=|x|(T-t)^{\beta},
\end{equation}
for some positive exponents $\alpha$, $\beta$ and a finite blow up time $T\in(0,\infty)$. Replacing this form into \eqref{eq1}, we obtain that the self-similar profiles $f(\xi)$ solve the following differential equation
\begin{equation}\label{SSODE}
(f^m)''(\xi)-\alpha f(\xi)+\beta\xi f'(\xi)+\xi^{\sigma}f(\xi)=0, \quad \xi\in[0,\infty),
\end{equation}
with exponents
\begin{equation}\label{SSexp}
\alpha=\frac{\sigma+2}{\sigma(m-1)}, \quad \beta=\frac{1}{\sigma}.
\end{equation}
Similarly as in our companion paper \cite{IS2}, we define below what kind of profiles $f$ we look for.
\begin{definition}\label{def1}
We say that a solution $f$ to the ordinary differential equation \eqref{SSODE} is a \textbf{good profile} is it fulfills one of the following alternatives with respect to its initial behavior:

\indent (P1) $f(0)=a>0$, $f'(0)=0$.

\indent (P2) $f(0)=0$, $(f^m)'(0)=0$.

A good profile $f$ is called a \textbf{good profile with interface} at some point $\xi_1\in(0,\infty)$ if
$$
f(\xi_1)=0, \qquad (f^m)'(\xi_1)=0, \qquad f>0 \ {\rm on} \
(\xi_1-\delta,\xi_1), \ {\rm for \ some \ } \delta>0.
$$
\end{definition}
With this definition, we will first prove
\begin{theorem}[Existence of good profiles with interface]\label{th.exist}
For any $\sigma>0$, there exists at least one good profile with interface $f$ to \eqref{SSODE}. In particular, for $\sigma=\sigma_*:=\sqrt{2(m+1)}$, there exists an explicit good profile with interface satisfying property (P2) in Definition \ref{def1}, namely
\begin{equation}\label{expl.prof}
f_*(\xi):=\xi^{2/(m-1)}\left(\frac{m-1}{2m(m+1)}-B\xi^{\sigma_*}\right)_+^{1/(m-1)}, \quad B=\frac{(m-1)^2}{m(\sigma_*+2)(m\sigma_*+m+1)}.
\end{equation}
\end{theorem}
We plot in Figure \ref{fig0} the explicit self-similar solution with profile $f_*$ defined by \eqref{expl.prof}, taken at different times.
\begin{figure}[ht!]
  \begin{center}
  \includegraphics[width=13.5cm,height=7.5cm]{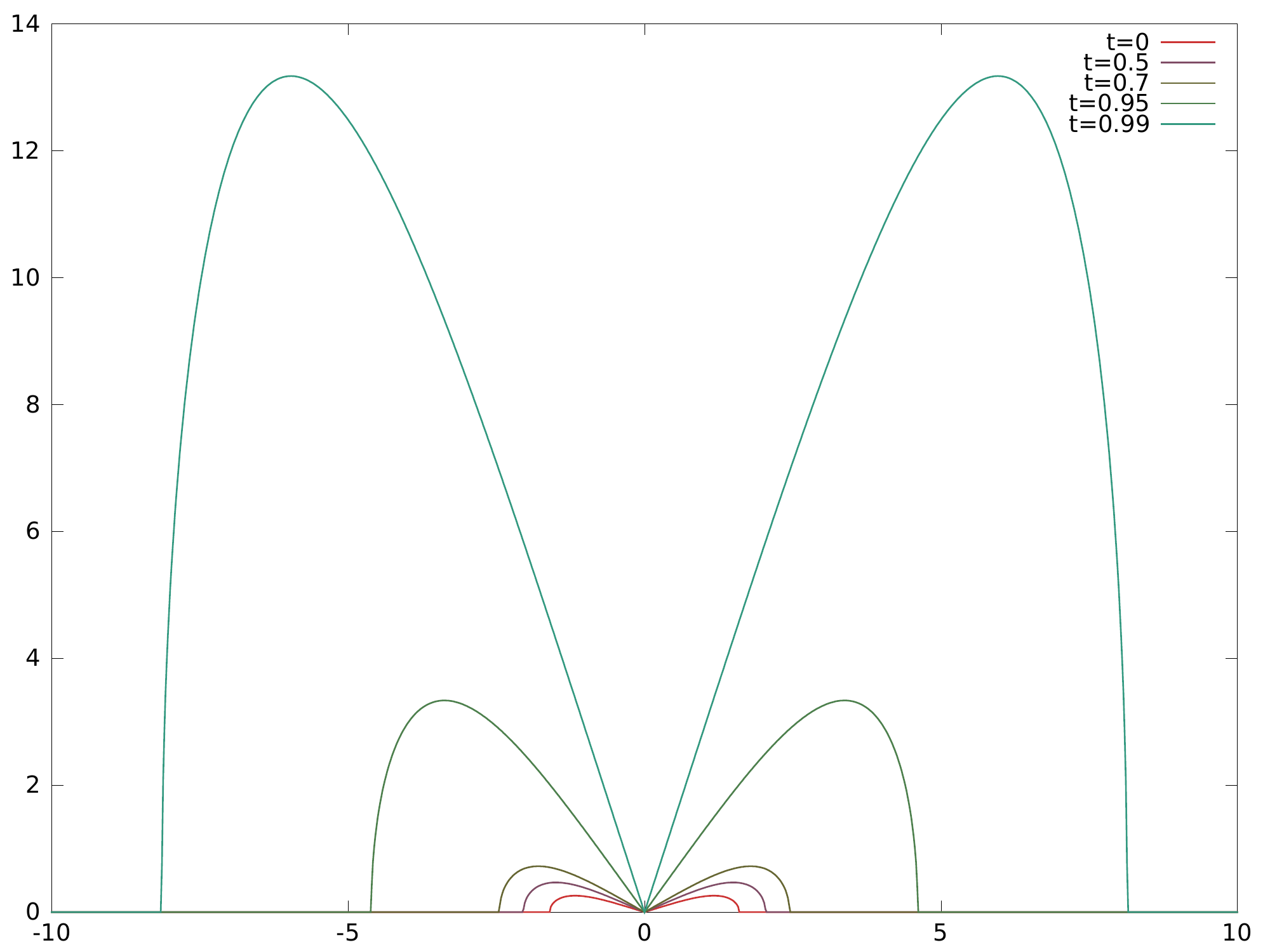}
  \end{center}
  \caption{The explicit self-similar solution with profile $f_*$. Experiment for $m=3$}\label{fig0}
\end{figure}
Apart from the mere existence of good solutions to our problem, Theorem \ref{th.exist} raises a very natural and interesting question: for which $\sigma>0$ there are good profiles satisfying property (P2) in Definition \ref{def1} and for which $\sigma$ (if any) there exist good profiles satisfying property (P1). In the more general case $1<p<m$ studied in \cite{IS2}, a partial answer to this question is easy to obtain: since for Eq. \eqref{eq1.gen} with $\sigma=0$ it was established in \cite[Theorem 2, p. 187]{S4} that good profiles satisfying (P1) exist, arguments based on continuity with respect to $\sigma$ show that at least in some interval $\sigma\in(0,\sigma_0)$ for some $\sigma_0>0$, there also exist good profiles satisfying property (P1) (see \cite[Theorem 1.3 and Section 4]{IS2}). But in the present case $p=1$, there is a very important detail: we \emph{cannot use the homogeneous case} $\sigma=0$ as a limit case, since (as already explained) in Eq. \eqref{eq1.hom} there is no finite time blow up. However, for $\sigma>0$ sufficiently small it is still true that all the good profiles with interface satisfy property (P1) in Definition \ref{def1}.
\begin{theorem}[Good profiles with interface for $\sigma>0$ small]\label{th.small}
There exists $\sigma_0>0$ sufficiently small such that, for any $\sigma\in(0,\sigma_0)$, all the good profiles with interface $f$ in the sense of Definition \ref{def1} satisfy
$$
f(0)=A>0, \quad f'(0)=0.
$$
\end{theorem}
With $\sigma$ sufficiently large, the behavior at $\xi=0$ of the good profiles with interface changes strongly. We already noticed that there exists an explicit profile with interface $f_*$ given in \eqref{expl.prof} which starts from $f(0)=0$. The next theorem gives a positive answer to the natural question whether there are more profiles of the same kind or not.
\begin{theorem}[Good profiles with interface for $\sigma>0$ large]\label{th.large}
There exists $\sigma_1>0$ sufficiently large such that, for any $\sigma\in(\sigma_1,\infty)$, there exists a good profile with interface $f$ in the sense of Definition \ref{def1} such that
\begin{equation}\label{buinf}
f(0)=0, \ \ (f^m)'(0)=0 \ \ {\rm and} \quad f(\xi)\sim C\xi^{(\sigma+2)/(m-1)}, \ \ {\rm as} \ \xi\to0.
\end{equation}
for some positive constant $C$.
\end{theorem}
In Figure \ref{fig1} we represent typical profiles in both cases $\sigma>0$ small (left, corresponding to the results of Theorem \ref{th.small}) and large (right, corresponding to the results of Theorem \ref{th.large}), performing a shooting from the interface point. This is in fact the way existence is proved, see Section \ref{sec.exist}.

\begin{figure}[ht!]
  \begin{center}
  \includegraphics[width=15cm,height=9cm]{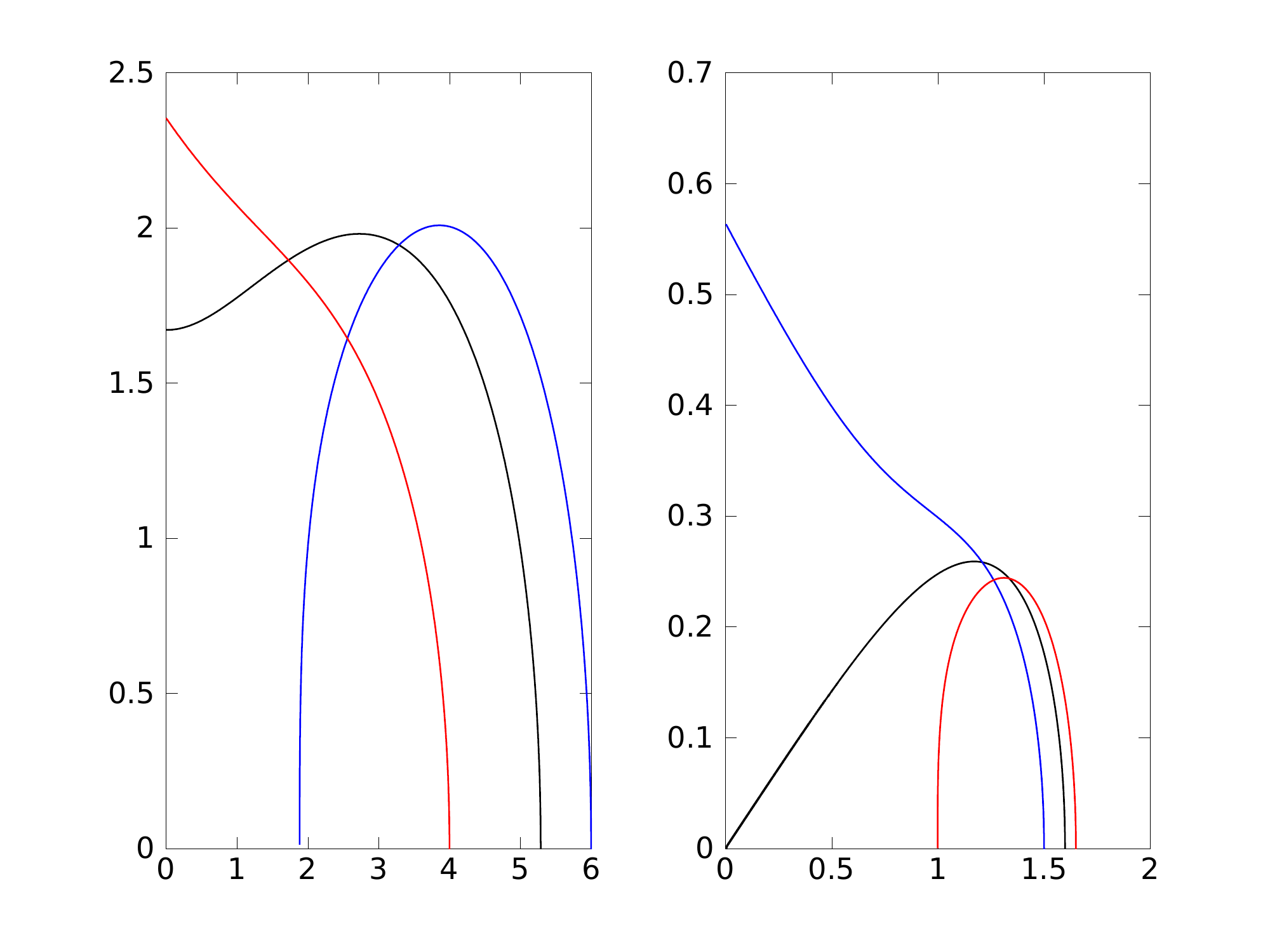}
  \end{center}
  \caption{Typical self-similar profiles for $\sigma$ small (left) and $\sigma$ large (right)}\label{fig1}
\end{figure}

\noindent \textbf{Remark.} Let us notice that the profile given by Theorem \ref{th.large} present a different behavior near $\xi=0$ than the explicit profile $f_*$. Indeed,
\begin{equation}\label{buzero}
f_*(\xi)\sim\left[\frac{m-1}{2m(m+1)}\right]^{1/(m-1)}\xi^{2/(m-1)}, \quad {\rm as} \ \xi\to0,
\end{equation}
which is a different first order power than in \eqref{buinf}. As we shall see below, these two behaviors are even qualitatively different concerning the blow up behavior. One very interesting conjecture is that $f_*$ is the unique good profile with interface having a local behavior near $\xi=0$ as in \eqref{buzero}.

\medskip

\noindent \textbf{Profiles with tails as $\xi\to\infty$.} Apart from the good profiles with interface, there is another bunch of interesting blow up profiles to \eqref{eq1}, that are positive everywhere and have a special tail at infinity, more precisely
\begin{equation}\label{tail}
f(\xi)\sim K\xi^{(\sigma+2)/(m-1)}e^{-\xi^{\sigma}}, \quad {\rm as} \ \xi\to\infty.
\end{equation}
These profiles have a spatial decay which is neither algebraic, nor purely exponential, but mixing an exponential decay with a precise algebraic power. Such profiles usually appear only in very special, critical cases (see for example the recent paper \cite{IL17} for another example of such special tail behavior for an absorption-diffusion equation with critical absorption exponent) but they do not exist for the homogeneous reaction-diffusion equation \eqref{eq1.hom}. More precisely, we prove
\begin{theorem}\label{th.tail}
For any $\sigma>0$, there exists at least a profile $f(\xi)$ solution to \eqref{SSODE} behaving as in \eqref{tail} as $\xi\to\infty$, and such that
$$
f(0)=0, \quad f(\xi)\sim C\xi^{(\sigma+2)/(m-1)}, \quad {\rm as} \ \xi\to0,
$$
for some $C>0$. Moreover, there exists $\sigma_0>0$ such that for any $\sigma\in(0,\sigma_0)$, \textbf{all} the good profiles \textbf{satisfying assumption (P2)} in Definition \ref{def1} present the tail behavior \eqref{tail} as $\xi\to\infty$.
\end{theorem}

\medskip

\noindent \textbf{Blow up behavior. Occurence of blow up at space infinity.} Having already classified the good self-similar profiles with interface, we can extract some very interesting conclusions concerning the blow up set of these profiles. It is obvious from their formula that the solutions to \eqref{eq1} with profiles as in Theorem \ref{th.small} present a global blow up. However, things are more complex when dealing with solutions to \eqref{eq1} having good profiles with interface satisfying assumption (P2) in Definition \ref{def1}. Take a solution having the form
$$
u(x,t)=(T-t)^{-\alpha}f(|x|(T-t)^{\beta}), \quad f(0)=0, \ (f^m)'(0)=0.
$$
Then for any $x\neq0$ fixed, $|x|(T-t)^{\beta}\to0$ as $t\to T$, that is, close to the blow up time $T\in(0,\infty)$, the behavior of $f(\xi)$ near $\xi=0$ becomes essential. For a generic solution $u$ to \eqref{eq1} with initial condition $u_0(x):=u(x,0)$ and blow up time $T\in(0,\infty)$, we define its \emph{blow up set} as in \cite[Section 24]{QS} by
\begin{equation}\label{BUS}
B(u_0):=\{x\in\real: \exists(x_k,t_k)\in\real\times(0,T), \ t_k\to T, \ x_k\to x, \ {\rm and} \  |u(x_k,t_k)|\to\infty, \ {\rm as} \ k\to\infty\}.
\end{equation}
We thus have two different cases:

$\bullet$ if $f(\xi)\sim C\xi^{2/(m-1)}$ as $\xi\to0$, as it is the case of the explicit profile \eqref{expl.prof}, then for $x\neq0$ fixed,
$$
u(x,t)\sim C(T-t)^{-\alpha+2\beta/(m-1)}|x|^{2/(m-1)}=C(T-t)^{-1/(m-1)}|x|^{2/(m-1)}, \quad {\rm as} \ t\to T,
$$
thus we infer that these solutions still blow up globally, that is $B(u_0)=\real$ according to the definition of the blow up set \eqref{BUS}.

$\bullet$ if $f(\xi)\sim C\xi^{(\sigma+2)/(m-1)}$ as $\xi\to0$, as it is the case of the good profiles with interface for $\sigma>\sigma_1$ in Theorem \ref{th.large}, then for $x\neq0$ fixed
$$
u(x,t)\sim C(T-t)^{-\alpha+(\sigma+2)\beta/(m-1)}|x|^{(\sigma+2)/(m-1)}=C|x|^{(\sigma+2)/(m-1)}<\infty, \quad {\rm as} \ t\to T,
$$
hence these solutions remain bounded at any finite point $x$. However, they still blow up at $t=T$, but only on curves $x(t)$ depending on $t$ such that $x(t)\to\infty$ as $t\to T$. This phenomenon is known in literature as \textbf{blow up at (space) infinity}, which seems to have been considered for the first time by Lacey \cite{La84}, and some other cases where it has been established (even for semilinear reaction-diffusion equation with big initial data) can be found in \cite{GU05, GU06}. Exactly the same considerations as above concerning the blow up behavior apply also to the profiles with tail behavior in Theorem \ref{th.tail}.

\noindent \textbf{Remark.} We notice that the blow up rate at a fixed point in the first of two cases above is $(T-t)^{-1/(m-1)}$, while the standard blow up rate, that is the blow up rate of $\|u(t)\|_{\infty}$ as defined in \cite[Section 23]{QS}, is $(T-t)^{-\alpha}$ in both cases above. This unusual difference is due to the fact that in both cases above
$$
M(t):=\max\{u(x,t):x\geq0\}\to\infty, \quad {\rm as} \ t\to T,
$$
thus no finite point has a similar blow up behavior as $\|u(t)\|_{\infty}=u(M(t),t)$.

We thus conclude that the weight $|x|^{\sigma}$ has also a \emph{very strong influence on the qualitative blow up behavior}, as seen for the previous special solutions (which are expected to be prototypes of general solutions): for $\sigma$ sufficiently small, blow up occurs globally, but starting from a sufficiently large exponent $\sigma$, the weight becomes so strong that it produces the above mentioned phenomenon of blow up at space infinity, that is, that the maximum (hot spot) $M(t)$ of the solution $u(x,t)$ at any time $t$ translates very quickly towards infinity both in space and time and thus blows up before any other (bounded) point may do it.

\medskip

\noindent \textbf{A comment about the notation.} It is clear that $\alpha$ and $\beta$ in \eqref{SSexp} depend on $\sigma$. For simplicity, we will usually drop this dependence from the notation when we refer to a general $\sigma>0$ and there is no danger of confusion. However, when the dependence on $\sigma$ is importance, we will emphasize it by writing $\alpha(\sigma)$ and $\beta(\sigma)$ only in the parts where this dependence is relevant. Moreover, we denote by $\alpha_*$ and $\beta_*$ the self-similar exponents corresponding to the special value $\sigma_*=\sqrt{2(m+1)}$.

\medskip

\noindent \textbf{Organization of the paper.} After the Introduction, we continue with a long and rather technical Section \ref{sec.local}, where the phase space technique for studying and classifying the solutions to \eqref{SSODE} is employed. The quadratic, autonomous dynamical system to which \eqref{SSODE} is transformed has many critical points, both in the finite region and at space infinity, and Section \ref{sec.local} is devoted to the local analysis of the behavior of solutions going out or entering all these critical points. We stress here that we deal in particular with one non-hyperbolic critical point at infinity which is very important for the study of the solutions and that requires special attention and a deeper work. We then pass to the proofs of our Theorems, by establishing (global) connections in the phase space and undoing the change of variables to obtain the desired profiles. Theorem \ref{th.exist} is proved in Section \ref{sec.exist}, which is rather short as the proof follows closely (and partly simplify) the analogous one in \cite[Section 3]{IS2}. The proof of Theorem \ref{th.small} is given in Section \ref{sec.small} and it is very elegant, based on constructing a sequence of suitable geometric barriers (in form of hypersurfaces and planes) for the flow in the phase space in order to "oblige" some orbits to have the desired behavior. Finally, Section \ref{sec.large} is devoted to the study of \eqref{SSODE} for $\sigma$ large, and Theorems \ref{th.large} and \ref{th.tail} are proved there, employing shooting arguments. The paper ends with a short section of comments and open problems.

\section{The phase space. Local analysis}\label{sec.local}

In this rather technical section, we transform the non-autonomous equation \eqref{SSODE} into an autonomous quadratic dynamical system of three equations and we perform the local analysis of the phase space associated to this system. The global analysis of the phase space, establishing the connections between the several critical points and transforming them into profiles solutions to \eqref{SSODE}, will be performed later in the paper. Recalling that $\alpha$ and $\beta$ are given in \eqref{SSexp}, we set
\begin{equation}\label{ph.space}
X(\eta)=\frac{m}{\alpha}\xi^{-2}f^{m-1}(\xi), \ \ Y(\eta)=\frac{m}{\alpha}\xi^{-1}f^{m-2}(\xi)f'(\xi), \ \ Z(\eta)=\frac{1}{\alpha}\xi^{\sigma},
\end{equation}
where the new independent variable $\eta=\eta(\xi)$ is defined through the following differential equation
$$
\frac{d\eta}{d\xi}=\frac{\alpha}{m}\xi f^{1-m}(\xi).
$$
In this notation, and after some rather straightforward algebraic calculations that are left to the reader, Eq. \eqref{SSODE} transforms into the following system
\begin{equation}\label{PSsyst}
\left\{\begin{array}{ll}\dot{X}=X[(m-1)Y-2X],\\
\dot{Y}=-Y^2-\frac{\beta}{\alpha}Y+X-XY-XZ,\\
\dot{Z}=\sigma ZX,\end{array}\right.
\end{equation}
Notice that $X\geq0$ and $Z\geq0$, only $Y$ can change sign. Moreover, let us also notice that the planes $\{X=0\}$ and $\{Z=0\}$ are invariant for the system \eqref{PSsyst}, and that $Z$ is monotonic along any connection in the phase space associated to \eqref{PSsyst}. These totally trivial facts will be of a great use when performing the global analysis.

\bigskip

\noindent \textbf{Local analysis of the finite critical points in the phase space}. We readily notice that the system \eqref{PSsyst} has the following critical points in the plane: two isolated points
$$
P_0=(0,0,0), \ \ P_2=\left(\frac{m-1}{2(m+1)\alpha},\frac{1}{(m+1)\alpha},0\right)
$$
and two lines of critical points denoted by
$$
P_1^{\gamma}=\left(0,-\frac{\beta}{\alpha},\gamma\right), \ \ P_0^{\gamma}=(0,0,\gamma), \ \ {\rm for \ any } \ \gamma>0.
$$
\begin{lemma}[Analysis of the point $P_0=(0,0,0)$]\label{lem.1}
The system \eqref{PSsyst} in a neighborhood of the critical point $P_0$ has a one-dimensional stable manifold and a two-dimensional center manifold. The connections in the plane tangent to the center manifold go out of the point $P_0$ and contain profiles with the behavior:
\begin{equation}\label{beh.P0}
f(\xi)\sim k\xi^{\alpha/\beta}=k\xi^{(\sigma+2)/(m-1)}, \quad {\rm as} \ \xi\to 0, \ f(0)=0,
\end{equation}
for any constant $k>0$.
\end{lemma}
\begin{proof}
The proof follows closely the one of Lemma 2.1 in \cite{IS2}. The linearization of the system \eqref{PSsyst} near $P_0$ has the matrix
$$
M(P_0)=\left(
      \begin{array}{ccc}
        0 & 0 & 0 \\
        1 & -\beta/\alpha & 0 \\
        0 & 0 & 0 \\
      \end{array}
    \right)
$$
hence it has a one-dimensional stable manifold (corresponding to the eigenvalue $-\beta/\alpha$) and a two-dimensional center manifold. Since we are interested in the orbits going out of $P_0$ in the phase space (if any), we will analyze this center manifold and the flow on it following the recipe given in \cite[Theorem 1, Section 2.12]{Pe}. Introducing the new variable
$$
W:=\frac{\beta}{\alpha}Y-X,
$$
we obtain after direct calculations (that we omit here) that the system \eqref{PSsyst} in the variables $(X,W,Z)$ becomes
\begin{equation}\label{interm0}
\left\{\begin{array}{ll}\dot{X}=(\sigma+2)XW+\sigma X^2,\\
\dot{W}=-\frac{m-1}{\sigma+2}W-\frac{\sigma+2}{m-1}W^2-\frac{m\sigma+3m+\sigma+1}{m-1}XW-\frac{m\sigma+m+1}{m-1}X^2-\frac{m-1}{\sigma+2}XZ,\\
\dot{Z}=\sigma XZ.\end{array}\right.
\end{equation}
We look for a center manifold of the type
$$
W=h(X,Z):=aX^2+bXZ+cZ^2+O(|(X,Z)|^3).
$$
Following the Local Center Manifold Theorem \cite[Theorem 1, Section 2.12]{Pe}, we have to identify only the quadratic terms in $X$ and $Z$ to deduce that the center manifold is given by
$$
h(X,Z)=-\frac{(m\sigma+m+1)(\sigma+2)}{(m-1)^2}X^2-XZ+O(|(X,Z)|^3),
$$
and the flow on the center manifold in a neighborhood of the critical point $P_0$ is given by the reduced system
$$
\left\{\begin{array}{ll}\dot{X}=\sigma X^2+O(|(X,Z)|^3),\\
\dot{Z}=\sigma XZ+O(|(X,Z)|^3),\end{array}\right.
$$
that can be readily integrated up to first order to give $Z\sim kX$ for the profiles going out of $P_0$ on the center manifold, where $k>0$ can be any constant. Taking into account the definitions of $X$ and $Z$ in \eqref{ph.space}, we infer that the profiles satisfy \eqref{beh.P0}, as claimed.
\end{proof}
Studying the behavior near the points $P_1^{\gamma}$ for $\gamma>0$ makes a difference with respect to the phase space used for the general case $1<p<m$ in \cite{IS2}. Indeed, while in our case we have a vertical line of critical points, for $p>1$ this line reduces to one point $P_1=P_1^0$.
\begin{lemma}[Analysis of the points $P_1^{\gamma}=(0,-\beta/\alpha,\gamma)$]\label{lem.2}
For any $\gamma>0$, the orbits entering the point $P_1^{\gamma}$ from the phase space contain profiles having an interface at a finite point:
\begin{equation}\label{beh.P1}
f(\xi)\sim\left(K(\gamma)-\frac{m-1}{2m\sigma}\xi^2\right)^{1/(m-1)}, \quad K(\gamma):=\frac{m-1}{2m\sigma}(\alpha\gamma)^{2/\sigma}>0,
\end{equation}
for $\xi\to\xi_0=(\alpha\gamma)^{1/\sigma}\in(0,\infty)$.
\end{lemma}
\begin{proof}
The linearization of the system \eqref{PSsyst} near $P_1^{\gamma}$ has the matrix
$$M(P_1^{\gamma})=\left(
  \begin{array}{ccc}
    -\frac{(m-1)\beta}{\alpha} & 0 & 0 \\
   1+\frac{\beta}{\alpha}-\gamma & \frac{\beta}{\alpha} & 0 \\
    \sigma\gamma & 0 & 0 \\
  \end{array}
\right),
$$
with eigenvalues
$$
\lambda_1=-\frac{(m-1)\beta}{\alpha}<0, \ \ \lambda_2=\frac{\beta}{\alpha}>0, \ \ \lambda_3=0
$$
and corresponding eigenvectors
$$
e_1=\left(-1,\frac{\alpha(1-\gamma)+\beta}{m\beta},\frac{\alpha\sigma\gamma}{(m-1)\beta}\right), \ \ e_2=(0,1,0), \ \ e_3=(0,0,1).
$$
According to \cite[Theorem 2.15, Chapter 9]{CH} and the Local Center Manifold Theorem \cite[Theorem 1, Section 2.10]{Pe}, it follows that all the center manifolds (recall that a priori the center manifold may not be unique) of dimension one in the neighborhood of the point $P_1^{\gamma}$ have to contain a segment of the invariant line $\{X=0, Y=-\beta/\alpha\}$. We thus readily deduce that the center manifold near $P_1^{\gamma}$ is unique. Similarly to the analysis in \cite{dPS02}, there exists only one trajectory entering $P_1^{\gamma}$ from outside the invariant plane $\{X=0\}$, tangent to the eigenvector $e_1$. All the other orbits are contained in the plane $\{X=0\}$ and thus do not contain profiles. The behavior of the profiles entering $P_1^{\gamma}$ is attained by equating
$$
Y\to-\frac{\beta}{\alpha}, \quad {\rm as} \ Z\to\gamma\in(0,\infty),
$$
which by integration gives \eqref{beh.P1}.
\end{proof}
\begin{lemma}[Analysis of the point $P_2=((m-1)/2(m+1)\alpha,1/(m+1)\alpha,0)$]\label{lem.3}
The system \eqref{PSsyst} in the neighborhood of the critical point $P_2$ has a two-dimensional stable manifold and a one-dimensional unstable manifold. The stable manifold is included in the invariant plane $\{Z=0\}$. There exists a unique orbit going out of $P_2$, containing profiles that locally satisfy
\begin{equation}\label{beh.P2}
f(0)=0, \quad f(\xi)\sim\left[\frac{m-1}{2m(m+1)}\right]^{1/(m-1)}\xi^{2/(m-1)}-\psi(\sigma)\xi^{(\sigma+2)/(m-1)}, \ \ {\rm as} \ \xi\to0,
\end{equation}
where $\psi(\sigma)$ is a coefficient depending on $\sigma$ such that $\psi(\sigma)\to0$ as $\sigma\to\infty$.
\end{lemma}
\begin{proof}
The linearization of the system \eqref{PSsyst} near the critical point $P_2$ has the matrix
$$
M(P_2)=\frac{1}{2(m+1)\alpha}\left(
  \begin{array}{ccc}
    -2(m-1) & (m-1)^2 & 0 \\
    2(m+1)\alpha-2 & -2\beta(m+1)-(m+3) & -(m-1) \\
    0 & 0 & \sigma(m-1) \\
  \end{array}
\right)
$$
with eigenvalues $\lambda_1$, $\lambda_2$ and $\lambda_3$ such that
$$
\lambda_1+\lambda_2=-\frac{3m+1+2\beta(m+1)}{2(m+1)\alpha}<0, \ \ \lambda_1\lambda_2=\frac{m-1}{2(m+1)\alpha^2}>0, \ \ \lambda_3=\frac{\sigma(m-1)}{2(m+1)\alpha},
$$
whence $\lambda_1$, $\lambda_2<0$. We are exactly in the same situation as in \cite[Lemma 2.3]{IS2} and the rest of the proof proceeds exactly as in the mentioned Lemma. In particular, all the orbits entering the point $P_2$ (tangent to the eigenvectors corresponding to the eigenvalues $\lambda_1$ and $\lambda_2$) lie in the invariant plane $\{Z=0\}$ and there is a unique orbit going out of $P_2$ tangent to the eigenvector corresponding to the positive eigenvalue $\lambda_3$, namely
\begin{equation}\label{eigen.P2}
e_3=\left(-\frac{(m-1)^2}{(m-1)\sigma^2+(3m+1)\sigma+4(m+1)},-\frac{(m-1)(\sigma+2)}{(m-1)\sigma^2+(3m+1)\sigma+4(m+1)},1\right).
\end{equation}
The behavior \eqref{beh.P2} is deduced similarly as in \cite[Lemma 2.3]{IS2}. At first, the first order term in \eqref{beh.P2} is obtained simply by noticing that $X(\xi)\to(m-1)/2(m+1)\alpha$ as $\xi\to0$ and using the definition of $X(\xi)$ in \eqref{ph.space}. For the second order term in \eqref{beh.P2}, we define:
$$
\psi(\sigma):=\left(\frac{(m-1)^2}{(m-1)\sigma^2+(3m+1)\sigma+4(m+1)}\right)^{1/(m-1)},
$$
and we further use the fact that the unique connection going out of $P_2$ lies tangent to the eigenvector $e_3$, thus (letting $X(P_2):=(m-1)/2\alpha(m+1)$)
$$
\frac{X-X(P_2)}{Z}\sim-\psi(\sigma)^{m-1}, \quad {\rm as} \ \xi\to0,
$$
or equivalently
$$
\frac{m}{\alpha}\left[f^{m-1}(\xi)-\frac{1}{2m(m+1)}\xi^2\right]\sim-\psi(\sigma)^{m-1}\xi^{\sigma+2}
$$
from where we readily deduce the second term in \eqref{beh.P2}.
\end{proof}
\begin{lemma}[Analysis of the point $P_0^{\gamma}=(0,0,\gamma)$]\label{lem.4}
For any $\gamma>0$, there is a unique orbit entering the point $P_0^{\gamma}$, which is contained in the $Z$ axis and do not contains interesting profiles.
\end{lemma}
\begin{proof}
The linearization of the system \eqref{PSsyst} near the critical point $P_0^{\gamma}$ for some $\gamma>0$ has the matrix
$$
M(P_0^{\gamma})=\left(
  \begin{array}{ccc}
    0 & 0 & 0 \\
    -\gamma+1 & -\frac{\beta}{\alpha} & 0 \\
    \sigma\gamma & 0 & 0 \\
  \end{array}
\right),
$$
thus presenting a two-dimensional center manifold and a one-dimensional stable manifold. One proceeds identically as in \cite[Lemma 2.4]{IS2} in order to analyze the center manifold by performing successively the changes of variable
$$
Z=\overline{Z}+\gamma, \quad G=\frac{\beta}{\alpha}Y-(1-\gamma)X,
$$
getting after rather tedious calculations (that are omitted here and left to the reader), the following system in the new variables:
\begin{equation}\label{center.syst}
\left\{\begin{array}{ll}\dot{X}=X[(\sigma+2)G-(\gamma(\sigma+2)-\sigma)X],\\
\dot{G}=-\frac{\beta}{\alpha}G-\frac{m\sigma+m+1-m(\sigma+2)\gamma}{m-1}(1-\gamma)X^2-\frac{\alpha}{\beta}G^2
-\frac{m-1+(\sigma+2)(m+1)(1-\gamma)}{m-1}XG-\frac{\beta}{\alpha}X\overline{Z},\\
\dot{\overline{Z}}=\sigma\gamma X+\sigma X\overline{Z}.\end{array}\right.
\end{equation}
We are now in a position to apply the Local Center Manifold Theorem \cite[Theorem 1, Section 2.12]{Pe} to the system \eqref{center.syst} and look for a center manifold of the form
$$
G(X,\overline{Z})=aX^2+bX\overline{Z}+c\overline{Z}^2+O(|(X,\overline{Z})|^3),
$$
with coefficients $a$, $b$ and $c$ to be determined according to the recipe given in the theorem. Keeping only the quadratic terms in the equation of the center manifold as in \cite[Theorem 1, Section 2.12]{Pe}, we finally get the equation of the center manifold
\begin{equation}\label{interm18}
G(X,\overline{Z})=\frac{\alpha}{\beta}\left[\sigma\gamma-(1-\gamma)\frac{m\sigma+m+1-m(\sigma+2)\gamma}{m-1}\right]X^2-X\overline{Z}+O(|(X,\overline{Z})|^3).
\end{equation}
In order to establish the flow on the center manifold, following again \cite[Theorem 1, Section 2.12]{Pe}, we replace the formula we obtained for $G(X,\overline{Z})$ into the equations for $\dot{X}$ and $\dot{\overline{Z}}$ in the system \eqref{center.syst}. We obtain that the flow on the center manifold near the point $P_{\gamma}$ is given by the system
\begin{equation}\label{flowgamma}
\left\{\begin{array}{ll}\dot{X}=X[(\sigma+2)G-(\gamma(\sigma+2)-\sigma)X],\\
\dot{\overline{Z}}=\sigma\gamma X+O(|(X,\overline{Z})|^2),\end{array}\right.
\end{equation}
For any $\gamma\neq\sigma/(\sigma+2)$, the term with $G=O(|(X,\overline{Z})|^2)$ in \eqref{flowgamma} can be discarded and the remaining system can be trivially integrated up to first order to show that there is no connection entering or going out of $P_{\gamma}$. For $\gamma=\sigma/(\sigma+2)$, the reduced system \eqref{flowgamma} becomes
\begin{equation}\label{flowgamma2}
\left\{\begin{array}{ll}\dot{X}=(\sigma+2)X(aX^2-X\overline{Z})+O(|(X,\overline{Z})|^4,\\
\dot{\overline{Z}}=\sigma\gamma X+O(|(X,\overline{Z})|^2),\end{array}\right.
\end{equation}
where $a$ is the coefficient of $X^2$ in the formula \eqref{interm18} of $G$. In the phase plane associated to the system \eqref{flowgamma2}, an orbit may enter or go out of the critical point $(X,\overline{Z})=(0,0)$ only in the region limited by the $X$ axis and the straight line of equation $aX-\overline{Z}=0$, due to the positive sign of the expression
\begin{equation}\label{interm19}
\frac{d\overline{Z}}{dX}=\frac{\sigma^2}{(\sigma+2)^2}\frac{1}{aX^2-X\overline{Z}}.
\end{equation}
But for $X$ and $\overline{Z}$ sufficiently small, we have by \eqref{interm19} that
\begin{equation}\label{interm20}
\frac{d\overline{Z}}{dX}>a
\end{equation}
in the above mentioned region. Assuming by contradiction that an orbit entering or going out of $(X,\overline{Z})=(0,0)$ exists, then the slope $d\overline{Z}/dX\in[0,a]$ at any point along the orbit, which contradicts \eqref{interm20}. It thus follows that there is no orbit connecting to any of the points $P_{\gamma}$.
\end{proof}
\begin{remark}
In the general case $1<p<m$ \cite[Lemma 2.4]{IS2}, a special attractor for a particular value of $\gamma=1/(p-1)$ existed, and this particular point was very important for the analysis. In our case, since $p=1$, this point no longer exists in the finite space, but it appears as a critical point at infinity (on the Poincar\'e hypersphere, see below).
\end{remark}

\medskip

\noindent \textbf{Local analysis of the critical points at the space infinity}. Together with the finite critical points analyzed below, in order to understand the global picture of the phase space associated to the system \eqref{PSsyst}, we have to analyze its critical points at the space infinity. To this end, we pass to the Poincar\'e hypersphere according to the theory in \cite[Section 3.10]{Pe}. We thus introduce new variables $(\overline{X},\overline{Y},\overline{Z},W)$ by
$$
X=\frac{\overline{X}}{W}, \ Y=\frac{\overline{Y}}{W}, \ Z=\frac{\overline{Z}}{W}
$$
and it follows from \cite[Theorem 4, Section 3.10]{Pe} that the critical points at infinity lie on the Poincar\'e hypersphere at points
$(\overline{X},\overline{Y},\overline{Z},0)$ where $\overline{X}^2+\overline{Y}^2+\overline{Z}^2=1$ and the following system is fulfilled:
\begin{equation}\label{Poincare1}
\left\{\begin{array}{ll}\overline{X}Q_2(\overline{X},\overline{Y},\overline{Z})-\overline{Y}P_2(\overline{X},\overline{Y},\overline{Z})=0,\\
\overline{X}R_2(\overline{X},\overline{Y},\overline{Z})-\overline{Z}P_2(\overline{X},\overline{Y},\overline{Z})=0,\\
\overline{Y}R_2(\overline{X},\overline{Y},\overline{Z})-\overline{Z}Q_2(\overline{X},\overline{Y},\overline{Z})=0,\end{array}\right.
\end{equation}
where $P_2$, $Q_2$ and $R_2$ are the homogeneous second degree parts of the terms in the right hand side of the system \eqref{PSsyst}, that is
\begin{equation*}
\begin{split}
&P_2(\overline{X},\overline{Y},\overline{Z})=\overline{X}[(m-1)\overline{Y}-2\overline{X}],\\
&Q_2(\overline{X},\overline{Y},\overline{Z})=-\overline{Y}^2-\overline{X}\overline{Y}-\overline{X}\overline{Z},\\
&R_2(\overline{X},\overline{Y},\overline{Z})=\sigma\overline{X}\overline{Z}
\end{split}
\end{equation*}
We thus find that the system \eqref{Poincare1} becomes
\begin{equation}\label{Poincare2}
\left\{\begin{array}{ll}\overline{X}[-m\overline{Y}^2+\overline{X}\overline{Y}-\overline{X}\overline{Z}]=0,\\
\overline{X}\overline{Z}[(\sigma+2)\overline{X}-(m-1)\overline{Y}]=0,\\
\overline{Z}(\overline{Y}^2+(\sigma+1)\overline{X}\overline{Y}+\overline{X}\overline{Z})=0,\end{array}\right.
\end{equation}
Taking into account that we are considering only points with coordinates $\overline{X}\geq0$ and $\overline{Z}\geq0$, we find the following critical points at infinity (on the Poincar\'e hypersphere):
$$
Q_1=(1,0,0,0), \ \ Q_{2,3}=(0,\pm1,0,0), \ \ Q_4=(0,0,1,0), \ \
Q_5=\left(\frac{m}{\sqrt{1+m^2}},\frac{1}{\sqrt{1+m^2}},0,0\right).
$$
We perform below the analysis of all these points, which goes in line with the similar section in \cite{IS2}, thus some technical points will be skipped.
\begin{lemma}\label{lem.5}
The critical point $Q_1=(1,0,0,0)$ in the Poincar\'e sphere is an unstable node. The orbits going out of this point into the finite part of the phase space associated to the system \eqref{PSsyst} contain profiles $f$ solutions to \eqref{SSODE} such that $f(0)=D>0$ and any possible value of $f'(0)$.
\end{lemma}
\begin{proof}
We apply part (a) of \cite[Theorem 5, Section 3.10]{Pe} to infer that the flow in a neighborhood of $Q_1$ is topologically equivalent to the flow in a neighborhood of the origin for the system
\begin{equation}\label{systinf1}
\left\{\begin{array}{ll}-\dot{y}=-y+z-w+my^2+\frac{\beta}{\alpha}yw,\\
-\dot{z}=-(\sigma+2)z+(m-1)yz,\\
-\dot{w}=-2w+(m-1)yw,\end{array}\right.
\end{equation}
where the minus sign has to be chosen according to the direction of the flow in the system \eqref{PSsyst}. It readily follows that $Q_1$ is an unstable node, the linearization of the system \eqref{systinf1} near the origin having positive eigenvalues 1, 2 and $\sigma+2$. In order to study the behavior of the profiles contained in the orbits going out of $Q_1$, we notice that
\begin{equation}\label{interm1}
\frac{dz}{dw}\sim\frac{\sigma+2}{2}\frac{z}{w}, \quad \frac{dy}{dw}\sim\frac{1}{2}\frac{y}{w}-\frac{1}{2}\frac{dz}{dw}+\frac{1}{2},
\end{equation}
whence by integration, we find $z\sim Cw^{(\sigma+2)/2}$. Coming back to the original variables and recalling that the projection of the Poincar\'e hypersphere has been done by dividing by the $X$ variable, we infer that $Z/X\sim KX^{-(\sigma+2)/2}$, and by direct integration,
$$
f(\xi)\sim K>0, \quad {\rm as} \ \xi\to0.
$$
Moreover, by integrating the second equivalence in \eqref{interm1}, we furthermore find
$$
Y\sim KX^{1/2}+1, \quad K\in\real,
$$
and by another integration step we obtain that either
\begin{equation}\label{beh.Q11}
f(\xi)\sim(C+K\xi)^{2/(m-1)}, \quad {\rm as} \ \xi\to0, \quad K\neq0, \ C\neq0
\end{equation}
or
\begin{equation}\label{beh.Q12}
f(\xi)\sim\left(C+\frac{\alpha(m-1)}{2m}\xi^2\right)^{1/(m-1)}.
\end{equation}
The two behaviors in \eqref{beh.Q11} and \eqref{beh.Q12} both have $f(0)>0$, but differ by their slope: indeed $f'(0)\neq0$ in \eqref{beh.Q11}, respectively $f'(0)=0$ in \eqref{beh.Q12}.
\end{proof}
\begin{lemma}\label{lem.6}
The critical points at infinity represented on the Poincar\'e hypersphere as $Q_{2,3}=(0,\pm1,0,0)$ are an unstable node, respectively a stable node. The orbits going out of $Q_2$ into the finite part of the phase space contain profiles $f(\xi)$ for which there exists $\xi_0\in(0,\infty)$ with $f(\xi_0)=0$, $f'(\xi_0)>0$. On the opposite, the profiles entering $Q_3$ from the finite part of the phase space contain profiles $f(\xi)$ for which there exists $\xi_0\in(0,\infty)$ with $f(\xi_0)=0$, $f'(\xi_0)<0$.
\end{lemma}
\begin{proof}
We apply part (b) of \cite[Theorem 5, Section 3.10]{Pe} to find that the flow of the system in a neighborhood of the points $Q_{2,3}$ is topologically equivalent to the flow defined by the system
\begin{equation}\label{systinf2}
\left\{\begin{array}{ll}\pm\dot{x}=-mx-\frac{\beta}{\alpha}xw+x^2+x^2w-x^2z,\\
\pm\dot{z}=-z-(\sigma+1)xz-\frac{\beta}{\alpha}zw+xzw-xz^2,\\
\pm\dot{w}=-w-\frac{\beta}{\alpha}w^2-xw+xw^2-xzw\end{array}\right.
\end{equation}
in a neighborhood of the origin. Moreover, since approaching the points $Q_{2,3}$ we have $|Y/X|\to\infty$ and $|Y/Z|\to\infty$, we deduce from the second equation of the system \eqref{PSsyst} that
$$
\dot{Y}=-Y^2-\frac{\beta}{\alpha}Y+X-XY-XZ<0,
$$
that is, $Y$ is decreasing in a neighborhood of the points $Q_{2,3}$. This shows that in the system \eqref{systinf2} we have to choose the minus sign for the point $Q_2=(0,1,0,0)$ and the plus sign for the point $Q_3=(0,-1,0,0)$. It readily follows then that $Q_2$ is an unstable node and $Q_3$ is a stable node. Moreover, we notice that the behavior near these points, in the variables of the system \eqref{systinf2}, is given by
$$
\frac{dx}{dw}\sim m\frac{x}{w},
$$
whence $x\sim Cw^m$ for some $C>0$. Coming back to the original variables and taking into account that the projection on the Poincar\'e hypersphere has been done by dividing by the $Y$ variable, we find
$$
\frac{X}{Y}\sim C\frac{1}{Y^m},
$$
whence by direct integration we find
\begin{equation}\label{beh.Q23}
f(\xi)\sim[K+C\xi^{2m/(m-1)}]^{1/m}, \quad K, \ C\in\real, \ K\neq0.
\end{equation}
We remark that for the orbits entering $Q_3$, $Y(\xi)<0$ in a neighborhood of it, which means that the profiles contained in such orbits have $f'<0$ and $C<0$, which implies $K>0$ in the formula \eqref{beh.Q23}. Thus there exists $\xi_0\in(0,\infty)$ such that $f(\xi_0)=0$, $f'(\xi_0)<0$. On the other hand, the profiles going out of $Q_2$ have $Y>0$ (although decreasing), hence $C>0$ in \eqref{beh.Q23}. Thus, for part of these profiles, more precisely those with $K<0$ in \eqref{beh.Q23}, there exists $\xi_0\in(0,\infty)$ such that $f(\xi_0)=0$, $f'(\xi_0)>0$, as stated.
\end{proof}
Before analyzing the behavior of the system near the critical point $Q_4$, we first perform the local analysis near the last point in the previous list $Q_5$.
\begin{lemma}\label{lem.8}
The critical point at infinity represented as $Q_5=(m/\sqrt{1+m^2},1/\sqrt{1+m^2},0,0)$ in the Poincar\'e hypersphere has a two-dimensional unstable manifold and a one-dimensional stable manifold. The orbits going out from this point into the finite region of the phase space contain profiles satisfying $$
f(0)=0, \ \quad f(\xi)\sim K\xi^{1/m} \ {\rm as} \  \xi\to0, \ K>0,
$$
in a right-neighborhood of $\xi=0$.
\end{lemma}
\begin{proof}
Using again the recipe given in \cite[Section 3.10]{Pe}, we infer that the flow in a neighborhood of the point $Q_5$ is topologically equivalent to the flow of the same system \eqref{systinf1} but this time in a neighborhood of the critical point with coordinates $(y,z,w)=(1/m,0,0)$. Moreover, when approaching $Q_5$, we have
$$
\frac{X}{Y}=\frac{\overline{X}}{\overline{Y}}\sim m,
$$
hence $X\sim mY$ in a (finite) neighborhood of $Q_5$ and
$$
\dot{X}=X[(m-1)Y-2X]\sim-m(m+1)Y^2<0,
$$
thus we have to choose again the minus sign in the system \eqref{systinf1}. The linearization of \eqref{systinf1} near $Q_5$ has the matrix
$$
M(Q_5)=\left(
         \begin{array}{ccc}
           1 & 1 & \frac{\beta}{m\alpha}-1 \\
           0 & -\frac{m\sigma+m+1}{m} & 0 \\
           0 & 0 & -\frac{m+1}{m} \\
         \end{array}
       \right),
$$
whence, taking into account the choice of the minus sign in front of the system \eqref{systinf1}, we obtain a two-dimensional unstable manifold and a
one-dimensional stable manifold. An easy analysis of the eigenvectors of the matrix $M(Q_5)$ shows that the orbits going out from $Q_5$ on the unstable manifold go to the finite part of the phase-space, while the orbits entering $Q_5$ on the stable manifold are contained in the hypersphere and thus not interesting for us. Passing now to profiles, we simply start from the relation $X\sim mY$ in a neighborhood of $Q_5$, whence
$$
\frac{m}{\alpha}f^{m-1}(\xi)\xi^{-2} \sim \frac{m^2}{\alpha}f^{m-2}(\xi)f'(\xi)\xi^{-1},
$$
and after direct integration we obtain
$$
f(\xi)\sim K\xi^{1/m} \ \ {\rm as} \ \xi\to0, \qquad {\rm for} \
K>0,
$$
as desired.
\end{proof}
As the reader probably noticed, we left apart up to now the local analysis of the critical point $Q_4$. The reason for it is that it is more involved, since the method in part (c) of \cite[Theorem 5, Section 3.10]{Pe} only leads to a purely quadratic system, with all eigenvalues zero. In order to simplify the study, we need first a preparatory result.
\begin{lemma}\label{lem.Q41}
Let $f(\xi)$ be a solution to \eqref{SSODE} such that $f(\xi)>0$ for any $\xi>0$. Then, there exists $R>0$ and $K=K(R)>0$ such that
\begin{equation}\label{beh.Q41}
f(\xi)\leq K\xi^{(\sigma+2)/(m-1)}e^{-\xi^{\sigma}}, \quad {\rm for \ any} \ \xi>R.
\end{equation}
\end{lemma}
\begin{proof}
We divide the proof into several steps.

\medskip

\noindent \textbf{Step A.} In a first step, let $\xi_0$ be a local minimum point for $f$ (and thus for $f^m$ at the same time), if any. Then $(f^m)''(\xi_0)\geq0$, $f'(\xi_0)=0$ and we infer from \eqref{SSODE} that
$$
\xi_0^{\sigma}f(\xi_0)\leq\alpha f(\xi_0),
$$
whence, taking into account that $f(\xi_0)>0$, we get $\xi_0\leq\alpha^{1/\sigma}$. This shows that a non-compactly supported profile $f(\xi)$ is necessarily monotonic for $\xi>R$ for some $R$ sufficiently large.

\medskip

\noindent \textbf{Step B.} In a second step, suppose for contradiction that $f(\xi)$ has a limit as $\xi\to\infty$, which is finite and nonzero, more precisely
$$
\lim\limits_{\xi\to\infty}f(\xi)=L\in(0,\infty).
$$
Then, since the same is true (with $L^m$) for the function $f^m(\xi)$ and since $f(\xi)$ converges to its limit $L$ either in an increasing way or in a decreasing way, we deduce by applying twice \cite[Lemma 2.9]{IL13a} (either to the function $f(\xi)-L$ if $f$ decreasing near infinity or to the function $L-f(\xi)$ if $f$ increasing near infinity) that there exists a subsequence $\{\xi_k\}_{k\geq1}$ such that
$$
\lim\limits_{k\to\infty}(f^m)''(\xi_k)=\lim\limits_{k\to\infty}\xi_kf'(\xi_k)=0, \quad \lim\limits_{k\to\infty}\xi_k=\infty.
$$
We furthermore obtain by evaluating \eqref{SSODE} at $\xi=\xi_k$ that
$$
\lim\limits_{k\to\infty}(\xi_k^{\sigma}-\alpha)f(\xi_k)=0,
$$
and a contradiction as it is easy to check that the above limit is in fact equal to $+\infty$.

\medskip

\noindent \textbf{Step C.} In a third step, suppose that $f(\xi)\to\infty$ as $\xi\to\infty$, which together with the first step, means that $f$ is increasing for $\xi>R$ with $R>0$ large. Thus $\beta\xi f'(\xi)>0$ for $\xi>R$ and it readily follows that
\begin{equation}\label{interm10}
\lim\limits_{\xi\to\infty}(f^m)''(\xi)=\lim\limits_{\xi\to\infty}(\alpha-\xi^{\sigma})f(\xi)=-\infty.
\end{equation}
Let $L>0$; by \eqref{interm10} there exists $\xi_L>R$ such that $(f^m)''(\xi)<-L$ for $\xi>\xi_L$. Applying Lagrange's theorem on a generic interval $(\xi_L,\xi)$ with $\xi>\xi_L$, we find
$$
(f^m)'(\xi)=(f^m)'(\xi_L)+(\xi-\xi_L)(f^m)''(\overline{\xi})<(f^m)'(\xi_L)-L(\xi-\xi_L), \quad \overline{\xi}\in(\xi_L,\xi),
$$
whence $(f^m)'(\xi)<0$ for $\xi$ sufficiently large, and a contradiction with the assumption that $f$ (and $f^m$) are increasing.

\medskip

\noindent \textbf{Step D.} Gathering the outcome of the three previous steps, we just proved that $f(\xi)\to0$ as $\xi\to\infty$ and it is decreasing on an interval $(R,+\infty)$. We first show that in this case $(f^m)''(\xi)\geq0$ at least for $\xi>R_1>R$ sufficiently large. Suppose that this is not true, thus there exists a sequence of intervals $\{[\xi_k^1,\xi_k^2]\}_{k\geq1}$ such that
$$
(f^m)''(\xi_k^1)=(f^m)''(\xi_k^2)=0, \ (f^m)''(\xi)<0 \ {\rm for} \ \xi\in(\xi_k^1,\xi_k^2), \quad \lim\limits_{k\to\infty}\xi_k^1=\lim\limits_{k\to\infty}\xi_k^2=+\infty.
$$
It is then easy to show that there exists at least a point $\xi_k\in(\xi_k^1,\xi_k^2)$ such that $(f^m)'(\xi_k)=0$ and $(f^m)''(\xi_k)<0$, hence $\xi_k$ is a point of local maxima for $f^m$, in contradiction to the fact that $f^m$ is decreasing to zero for $\xi\in(R,\infty)$. Thus, $(f^m)''(\xi)\geq0$ for $\xi$ large and it follows from \eqref{SSODE} that
\begin{equation}\label{interm11}
(\xi^{\sigma}-\alpha)f(\xi)+\beta\xi f'(\xi)\leq 0, \quad \xi\in(R_1,\infty).
\end{equation}
The inequality \eqref{interm11} can be written equivalently as
$$
\frac{f'(\xi)}{f(\xi)}\leq-\frac{\xi^{\sigma}-\alpha}{\beta\xi}=-\sigma\xi^{\sigma-1}+\frac{\sigma+2}{m-1}\frac{1}{\xi},
$$
which can be directly integrated over a generic interval $(\xi_0,\xi)$ for some $\xi_0>R_1$ fixed to obtain that
$$
f(\xi)\leq\left[e^{\xi_0^{\sigma}}f(\xi_0)\xi_0^{-(\sigma+2)/(m-1)}\right]\xi^{(\sigma+2)/(m-1)}e^{-\xi^{\sigma}},
$$
which leads to \eqref{beh.Q41}.
\end{proof}
Coming back to the phase space associated to the system \eqref{PSsyst}, Lemma \ref{lem.Q41} implies in particular that $XZ\to0$ as $\xi\to\infty$ along all the profiles $f$ that are positive everywhere. This allows us to complete the local analysis near the critical point $Q_4$.
\begin{lemma}\label{lem.Q42}
The critical point $Q_4=(0,0,1,0)$ on the Poincar\'e hypersphere is a non-hyperbolic critical point of the system \eqref{PSsyst}, which behaves like an attractor for the connections coming from the finite part of the phase space associated to the system \eqref{PSsyst}. All the profiles contained in orbits entering this point have a \emph{tail at the space infinity} of the form
\begin{equation}\label{beh.Q42}
f(\xi)\sim K\xi^{(\sigma+2)/(m-1)}e^{-\xi^{\sigma}}, \quad {\rm as} \ \xi\to\infty,
\end{equation}
for some constant $K>0$.
\end{lemma}
\begin{proof}
We perform the change of variable $W:=XZ$ in the system \eqref{PSsyst}, obtaining the following system
\begin{equation}\label{PSsyst2}
\left\{\begin{array}{ll}\dot{X}=X[(m-1)Y-2X],\\
\dot{Y}=-Y^2-\frac{\beta}{\alpha}Y+X-XY-W,\\
\dot{W}=W[(m-1)Y+(\sigma-2)X].\end{array}\right.
\end{equation}
Notice that Lemma \ref{lem.Q41} implies that for a non-compactly supported profile $f$ we have $X(\xi)=(m/\alpha)\xi^{-2}f^{m-1}(\xi)\to0$, $Y(\xi)=(m/\alpha)\xi^{-1}(f^{m-1})'(\xi)\to0$ and $W(\xi)=(XZ)(\xi)\to0$ as $\xi\to\infty$, thus any profile contained in an orbit of \eqref{PSsyst} connecting to $Q_4$ transforms into a profile contained in an orbit of the new system \eqref{PSsyst2} connecting to the origin of it. It thus only remains to analyze locally the flow of the system \eqref{PSsyst2} near the critical point $(X,Y,W)=(0,0,0)$. The linearization of the system \eqref{PSsyst2} near the point $(0,0,0)$ has the matrix
$$M(Q_4)=\left(
         \begin{array}{ccc}
           0 & 0 & 0 \\
           1 & -\frac{\beta}{\alpha} & -1 \\
           0 & 0 & 0 \\
         \end{array}
       \right),
$$
hence near this point we have a one-dimensional stable manifold and a two-dimensional center manifold. In order to study the center manifold, we first do the change of variable
$$
T:=-X+\frac{\beta}{\alpha}Y+W,
$$
and the system \eqref{PSsyst2} becomes (after rather long calculations):
\begin{equation}\label{interm12}
\left\{\begin{array}{ll}\dot{X}=X[(\sigma+2)T+\sigma X-(\sigma+2)W],\\
\dot{T}=-\frac{\beta}{\alpha}T-\frac{m\sigma+m+1}{m-1}X^2-\frac{m(\sigma+2)}{m-1}W^2+\frac{3(\sigma+1)(m-1)+2(\sigma+2)}{m-1}XW+o(T),\\
\dot{W}=W[(\sigma+2)T+2\sigma X-(\sigma+2)W].\end{array}\right.
\end{equation}
We use again the Local Center Manifold Theorem \cite[Theorem 1, Section 2.12]{Pe} and we look for a two-dimensional center manifold of the form
$$
T(X,W)=aX^2+bXW+cW^2+O(|(X,W)|^3),
$$
with coefficients $a$, $b$ and $c$ to be determined. We next readily identify the coefficients by taking into account only the quadratic terms in $X$ and $W$ in the system \eqref{interm12} and deduce the equation of the center manifold:
\begin{equation*}
\begin{split}
T(X,W)&=-\frac{\sigma+2}{m-1}\left[\frac{m\sigma+m+1}{m-1}X^2+\frac{m(\sigma+2)}{m-1}W^2-\frac{3(\sigma+1)(m-1)+2(\sigma+2)}{m-1}XW\right]\\&+O(|(X,W)|^3).
\end{split}
\end{equation*}
Moreover, we readily deduce by using again the Local Center Manifold Theorem \cite[Theorem 1, Section 2.12]{Pe} and simply neglecting the terms in $R$ (that are of higher order) from the equations for $\dot{X}$ and $\dot{W}$ in \eqref{interm12} that the flow on the center manifold is given by the reduced system
\begin{equation}\label{interm13}
\left\{\begin{array}{ll}\dot{X}=X[\sigma X-(\sigma+2)W],\\
\dot{W}=W[2\sigma X-(\sigma+2)W],\end{array}\right.
\end{equation}
whose flow near the non-hyperbolic critical point $(X,W)=(0,0)$ has been classified in the paper \cite{Ly51}. Indeed, system \eqref{interm13} also writes
\begin{equation}\label{interm15}
\frac{dW}{dX}=f\left(\frac{W}{X}\right), \quad f(k)=\frac{2\sigma k-(\sigma+2)k^2}{\sigma-(\sigma+2)k},
\end{equation}
which is of the general form $f(k)=(a+bk+ck^2)/(d+ek+fk^2)$ as in \cite{Ly51}, with coefficients $a=f=0$, $c=e=-(\sigma+2)$, $b=2\sigma$ and $d=\sigma$. With this notation we notice that
$$
f'(0)=\frac{b}{d}=2>1, \ \ \lim\limits_{k\to\infty}f'(k)=\frac{c}{e}=1, \ \ \frac{b-d}{e}=-\frac{\sigma}{\sigma+2}<0,
$$
thus we are in the case of the phase portrait number 7 in \cite[p. 176]{Ly51} but having to interchange the quadrants, as in the mentioned paper the author assumes without loss of generality that $(b-d)/e>0$. Thus, the phase portrait of the flow near the origin of the system \eqref{interm13} in the first quadrant (since we only work with $X\geq0$ and $W=XZ\geq0$) is given by an elliptic sector with cycles formed by connections going out of the origin on the center manifold and also entering the origin tangent to the same center manifold, see Figure \ref{fig2} below (the reader can also see \cite[Case 7, p. 176]{Ly51} or also \cite[Figure 5, p. 152, Section 2.11]{Pe}).

\begin{figure}[ht!]
  \begin{center}
  \includegraphics[width=9cm,height=8cm]{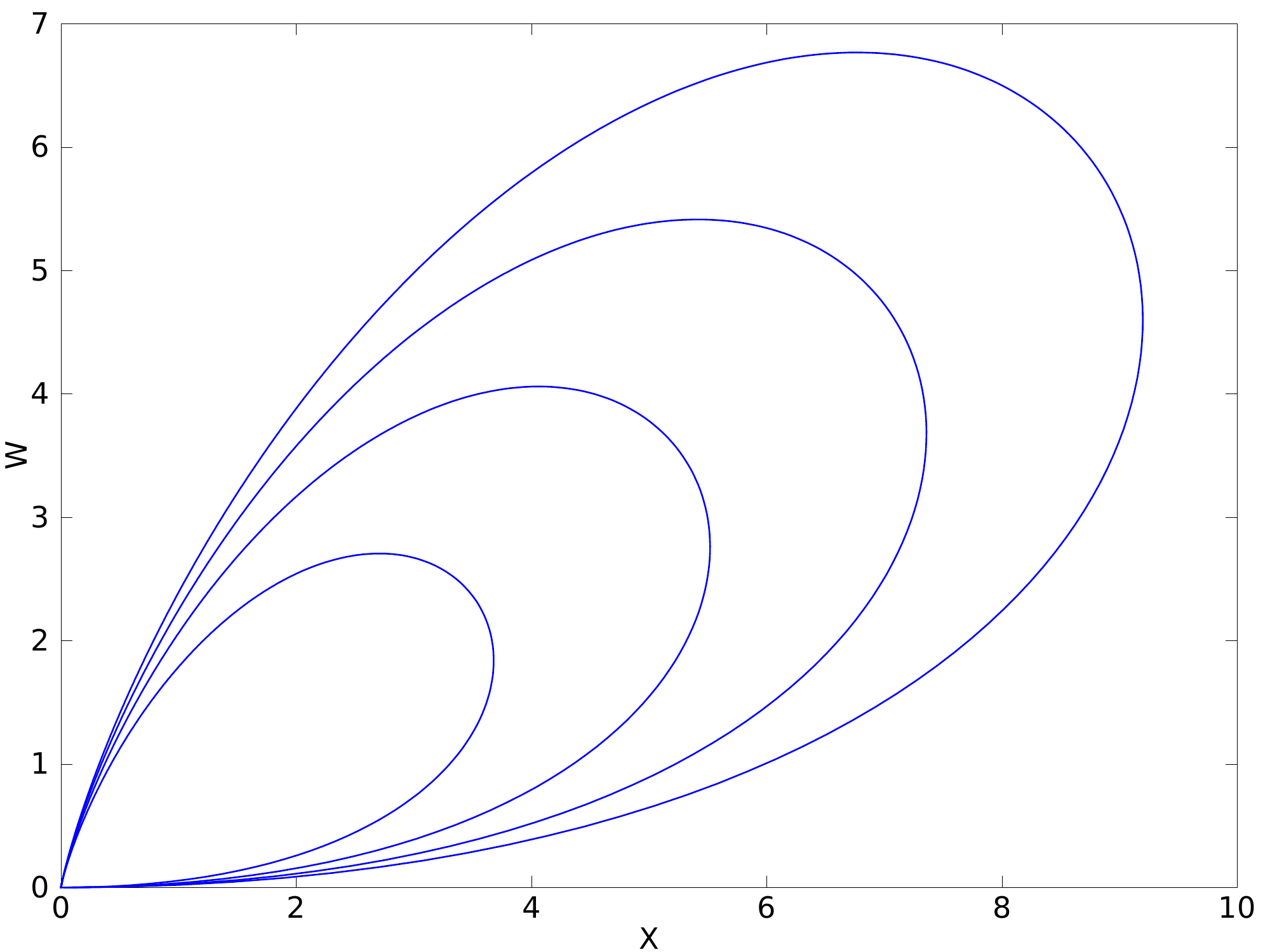}
  \end{center}
  \caption{Local behavior of the system \eqref{interm13} with the elliptic sector at the origin.}\label{fig2}
\end{figure}

Since by definition $W=XZ$, the orbits going out of the origin in the system \eqref{interm13} inherit the behavior of the connections going out of the point $P_0$ in the system \eqref{PSsyst}, established in Lemma \ref{lem.1}. That is, they contain profiles such that
$$
f(0)=0, \ f(\xi)\sim k\xi^{(\sigma+2)/(m-1)}, \quad {\rm as} \ \xi\to0.
$$
On the other hand, analyzing the connections that come from the orbits of the system \eqref{PSsyst} with $Z\to\infty$, it follows that in a neighborhood of the critical point $(X,W)=(0,0)$ of the system \eqref{interm13} $W$ dominates over $X$, so that $\dot{X}$, $\dot{W}<0$ and these are thus the orbits entering the origin in the system \eqref{interm13}, which connect to the ones going out from the same critical point forming an elliptic sector, as explained above. Finally, we integrate the system \eqref{interm13} focusing only on the behavior of the connections entering the origin (that is, as $\xi\to\infty$). We first notice that
$$
\frac{dW}{dX}=\frac{W}{X}\frac{2\sigma X-(\sigma+2)W}{\sigma X-(\sigma+2)W},
$$
and letting $W=W(X):=XK(X)$, we further obtain that
$$
K'(X)=\frac{\sigma K(X)}{X[\sigma-(\sigma+2)K(X)]},
$$
which can be integrated explicitly to get
\begin{equation}\label{interm14}
cK(X)\exp\left(-\frac{\sigma+2}{\sigma}K(X)\right)=X, \quad c>0 \ {\rm arbitrary}.
\end{equation}
On the other hand, coming back to the initial variables and recalling that $W=XZ$, we find that in fact $K=Z$ and replacing in \eqref{interm14} the expressions of $X$, $Z$ in terms of profiles from \eqref{ph.space} we finally get
$$
\frac{m}{\alpha}\xi^{-2}f^{m-1}(\xi)=c\frac{1}{\alpha}\xi^{\sigma}\exp\left(-\frac{\sigma+2}{\sigma}\frac{\sigma(m-1)}{\sigma+2}\xi^{\sigma}\right),
$$
and we easily infer that $f(\xi)$ behaves as stated in \eqref{beh.Q42}. Moreover, recalling that the orbits reaching the point $Q_4$ (that is, with $Z\to\infty$) were mapped via the change of variable $W=XZ$ into the orbits entering the origin over the center manifold in the system \eqref{PSsyst2}, we readily infer from the Center Manifold Theorem and from the form of the elliptic sector near the origin of the system \eqref{interm13} that $Q_4$ behaves like an attractor for the orbits coming from the finite part of the phase space associated to the initial system \eqref{PSsyst}. More precisely, since small neighborhoods of $Q_4$ in the finite region of the phase space $Z<\infty$ are mapped into neighborhoods of the origin of the system \eqref{PSsyst2} via the continuous change of variable $Z\mapsto W=XZ$, it follows that there is an open finite-neighborhood (that is, lying inside the region $Z<\infty$) of $Q_4$ such that, if an orbit enters that open neighborhood, it then necessarily enters $Q_4$.
\end{proof}

\section{Existence of good profiles with interface}\label{sec.exist}

This section is devoted to the proof of Theorem \ref{th.exist}. It will follow the same steps as the proof of the analogous Theorem 1.2 in \cite[Section 3]{IS2}, thus some technical parts will be omitted here if they are identical to the proof in the above mentioned section. However, since we deal with the critical case $p=1$, some alternative, simpler proofs of the preliminary results can be given. As a preliminary result, similar to \cite[Proposition 3.3]{IS2}, we have the following
\begin{lemma}\label{lem.uniq}
For any $\xi_0\in(0,\infty)$, there exists a unique profile $f(\xi)$ solution to \eqref{SSODE} such that
$$
f(\xi)>0 \ {\rm for} \ \xi\in(0,\xi_0), \qquad f(\xi)=0 \ {\rm for}
\ \xi\geq\xi_0, \qquad \lim\limits_{\xi\to\xi_0}(f^m)'(\xi)=0,
$$
that is, a unique profile having an interface exactly at $\xi=\xi_0$.
\end{lemma}
\begin{proof}
The conclusion follows readily from the local analysis of the points $P_1^{\gamma}$ for any $\gamma>0$ performed in Lemma \ref{lem.2} together with the definition of $Z=\xi^{\sigma}/\alpha$ realizing a one-to-one matching of interface points $\xi_0\in(0,\infty)$ and critical points $P_1^{\gamma}$ with $Z=\gamma=\xi_0^{\sigma}/\alpha$.
\end{proof}
The proof of Theorem \ref{th.exist} is, analogously as in \cite[Section 3]{IS2}, based on two independent propositions which establish how the (unique) profile with interface at a given point $\eta\in(0,\infty)$ (not necessarily a good profile in the sense of Definition \ref{def1}) behaves when either $\eta$ is very large, or very close to zero. This is another example of employment of the \emph{backward shooting method from the interface}, also used in our previous work \cite{IS2} and for the first time, up to our knowledge, in \cite{GP76}. We give below the statements and the proofs of these two propositions. Along all this section, $\sigma>0$ is fixed.
\begin{proposition}\label{prop.far}
There exists $\eta_f>0$ such that for any $\eta>\eta_f$, the unique profile $f_{\eta}$ solution to \eqref{SSODE} with interface point at $\xi=\eta$ changes sign backwards at some positive point, that is, there exists $\theta>0$ such that
$$
f_{\eta}(\theta)=0, \quad (f_{\eta}^m)'(\theta)>0, \quad f_{\eta}(\xi)>0 \ {\rm for} \ \xi\in(\theta,\eta).
$$
\end{proposition}
\begin{proof}
We come back to the system \eqref{PSsyst2} used in Lemma \ref{lem.Q42} by replacing the variable $Z$ by $W=XZ$. Restricting ourselves to the invariant plane $X=0$, we obtain the reduced system
\begin{equation}\label{interm16}
\left\{\begin{array}{ll}\dot{Y}=-Y^2-\frac{\beta}{\alpha}Y-W,\\
\dot{W}=(m-1)YW.\end{array}\right.
\end{equation}
Let us notice that by letting $W=XZ$, the interface points $P_1^{\gamma}$ of the system \eqref{PSsyst} gather into a single critical point $P_1:=(0,-\beta/\alpha,0)$ of the system \eqref{PSsyst2}, which in particular reduces to the point of coordinates $(-\beta/\alpha,0)$ for the system \eqref{interm16}. It is easy to check that this is a saddle point in the phase plane and thus there is only one orbit entering $P_1$ inside the invariant plane $\{X=0\}$. The fact that this orbit entering $P_1$ inside the plane $\{X=0\}$ comes from the unstable node $Q_2$ at infinity is a rather easy fact and has been proved with full details in \cite[Proposition 3.4, Step 1]{IS2}, to which we refer the reader. A visual representation of the phase plane associated to the system \eqref{interm16} in the invariant plane $\{X=0\}$, emphasizing on the unique connection entering $P_1$, is given in Figure \ref{fig3} below.

\begin{figure}[ht!]
  \begin{center}
  \includegraphics[width=12cm,height=9cm]{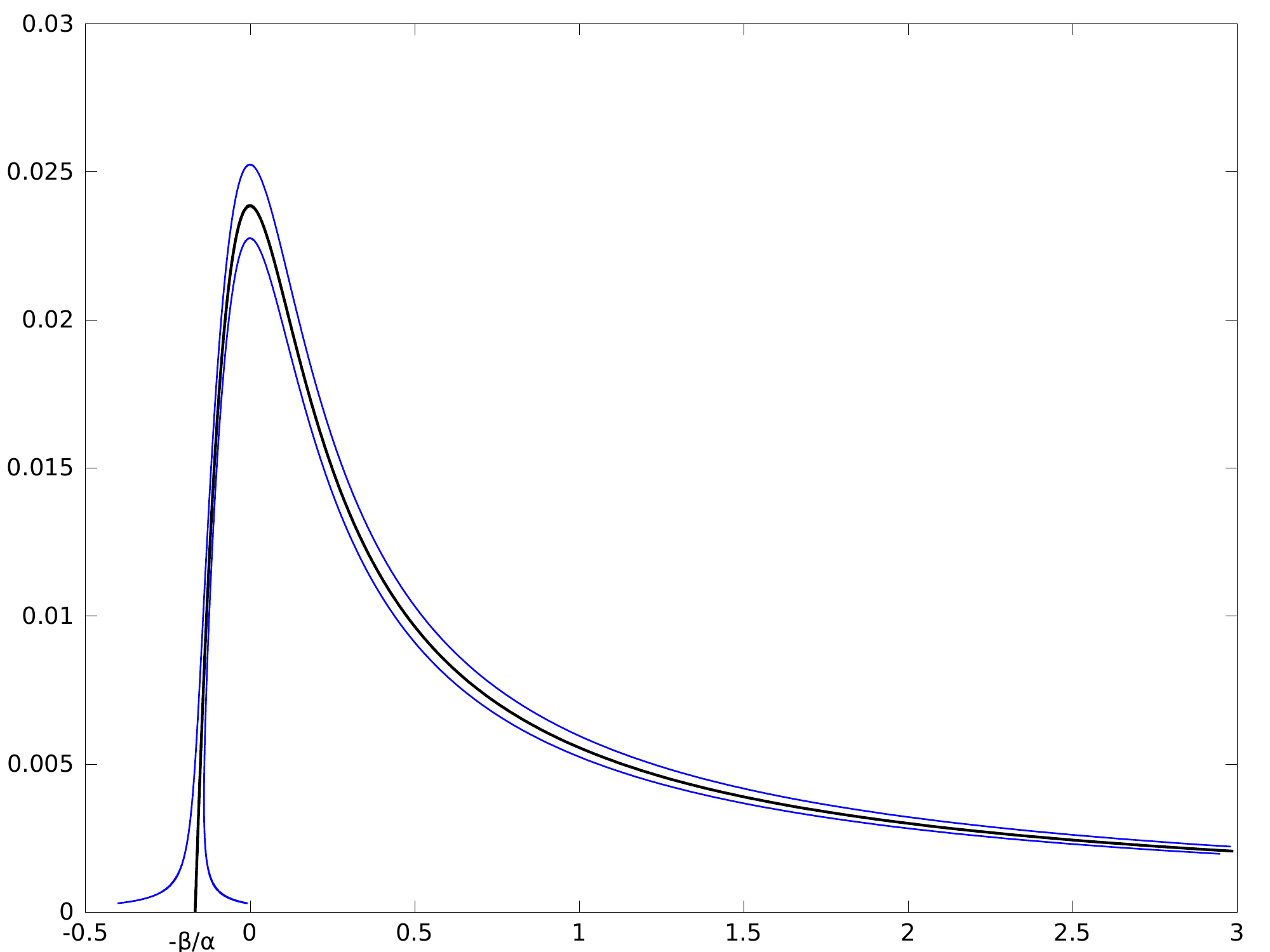}
  \end{center}
  \caption{Orbit entering $P_1$ in the $(Y,W)$ phase-plane associated to the system \eqref{interm16}.} \label{fig3}
\end{figure}

Since $Q_2$ is an unstable node, we obtain by standard continuity arguments that for any $\e>0$, there exists a connection in the phase space associated to the system \eqref{PSsyst2} connecting $Q_2$ to $P_1$ and such that $X(\xi)<\e$ for any $\xi\in(0,\infty)$.

Consider now the plane $\{Y=-\beta/2\alpha\}$. The direction of the flow of the system \eqref{PSsyst2} on this plane is given by the sign of the expression
$$
-\frac{\beta^2}{4\alpha^2}+\frac{\beta^2}{2\alpha^2}+X+\frac{\beta}{2\alpha}X-W=\frac{\beta^2}{4\alpha^2}+X\left(1+\frac{\beta}{2\alpha}\right)-W,
$$
and since the connection coming from $Q_2$ and entering $P_1$ should cross this plane at some point, it can do this only at points where the above expression is negative. This means that at the crossing point, we have
$$
W\geq X\left(1+\frac{\beta}{2\alpha}\right)+\frac{\beta^2}{4\alpha^2}\geq\frac{\beta^2}{4\alpha^2},
$$
and coming back to the initial system \eqref{PSsyst} and taking into account that $X<\e$ along the whole connection, we furthermore get
$$
Z=\frac{W}{X}\geq\frac{W}{\e}\geq\frac{\beta^2}{4\e\alpha^2}.
$$
This shows that along the orbit we consider (with $X(\xi)<\e$ everywhere) we have at least a point where
\begin{equation}\label{interm17}
\xi=(\alpha Z)^{1/\sigma}\geq\left[\frac{\beta^2}{4\e\alpha}\right]^{1/\sigma},
\end{equation}
and the right hand side of \eqref{interm17} can be done as large as we want by letting $\e>0$ very small. Since the interface point of the same connection is more to the right, so that it also satisfies \eqref{interm17}, it follows that any connection containing profiles with very large interface point comes from $Q_2$, reaching the conclusion.
\end{proof}
\begin{proposition}\label{prop.close}
There exists $\eta_c>0$ such that for any $\eta\in(0,\eta_c)$, the unique profile $f_{\eta}$ solution to \eqref{SSODE} with interface point at $\xi=\eta$ is decreasing in $(0,\eta)$.
\end{proposition}
As it is a rather trivial fact (see Lemma 3.1 in \cite[Section 3]{IS2} for details) that a good profile $f$ with $f(0)=A>0$, $f'(0)=0$ is increasing in a right-neighborhood of $\xi=0$, the two propositions above prove that the good profiles with interface (if any) can have their interface only inside a compact interval of $(0,\infty)$.
\begin{proof}[Proof of Proposition \ref{prop.close}]
The direction of the flow on the plane $\{Y=0\}$ in the phase space associated to the system \eqref{PSsyst} is given by the sign of $X(1-Z)$, which is negative for $Z>1$, thus a connection in the phase space can cross the plane $\{Y=0\}$ from the positive to the negative side only at points with $Z>1$. Let then $\gamma\in(0,1)$. The orbit entering the critical point $P_1^{\gamma}$ cannot then cross $\{Y=0\}$, as $Z$ is increasing along any orbit (and thus $0\leq Z<\gamma<1$ along it), hence the profile contained in it is decreasing. But this profile entering the point $P_1^{\gamma}$ has interface at point
$$
\xi=(\alpha\gamma)^{1/\sigma}\in(0,\alpha^{1/\sigma}),
$$
and we reach the conclusion by letting $\eta_c:=\alpha^{1/\sigma}$.
\end{proof}
\noindent \textbf{Remark.} The outcome of Propositions \ref{prop.far} and \ref{prop.close} is visually exemplified in Figure \ref{fig1} where we plot profiles with large and, respectively, small (close to zero) interface points.
\begin{proof}[Proof of Theorem \ref{th.exist}]
\noindent \textbf{Existence.} This part follows the same lines as the proof of the analogous result in \cite[Section 3]{IS2}, thus we give here a sketch for the sake of completeness. Let $A$ be the set of interface points $\eta\in(0,\infty)$ such that the unique profile $f_{\eta}$ with interface exactly at $\xi=\eta$ satisfies $f_{\eta}(0)=A>0$ and $f_{\eta}'(0)<0$ and let $B$ be the set of interface points $\eta\in(0,\infty)$ sufficiently big such that the unique profile $f_{\eta}$ with interface exactly at $\xi=\eta$ behaves like in Proposition \ref{prop.far}. It follows from the continuity with respect to parameters that both $A$ and $B$ are open sets; moreover, Propositions \ref{prop.close} and \ref{prop.far} insure that $A\neq\emptyset$, $B\neq\emptyset$ and $B$ contains an unbounded interval $(\eta^*,\infty)$. Let then
$$
\eta_0=\sup A<\eta^*<\infty.
$$
We claim that the good profile with interface we look for is $f_{\eta_0}$, the only profile with interface exactly at $\eta=\eta_0$. To prove this claim (and the theorem), we show that

$\bullet$ $f_{\eta_{0}}$ cannot have a vertical asymptote at $\xi=0$. This is just a calculus exercise, based on the fact that for any generic real function $f$ having a vertical asymptote at $\xi=0$, we have
$$
\lim\limits_{\xi\to0, \xi>0}\frac{f'(\xi)}{f(\xi)}=-\infty,
$$
see \cite[Lemma 3.6]{IS2}. This fact can be applied twice to the function $f_{\eta_0}^m$ to get that
\begin{equation*}
\lim\limits_{\xi\to0,\xi>0}\frac{(f_{\eta_0}^m)''(\xi)}{f_{\eta_0}'(\xi)}=-\infty,
\end{equation*}
which readily leads to a contradiction with Eq. \eqref{SSODE}, since $(f_{\eta_0}^m)''$ dominates over the other terms.

$\bullet$ There is no $\eta_1>0$ such that $f_{\eta_0}(\eta_1)=0$, $f_{\eta_0}(\xi)>0$ for $\xi\in(\eta_0,\eta_1)$. Indeed from the local analysis, if such point $\eta_1$ exists, then $f'(\eta_1)>0$ and thus $\eta_0\in B$, as $f_{\eta_0}$ would behave like the profiles in Proposition \ref{prop.far}. But since $B$ is open, there exists $\e>0$ such that $(\eta_0-\e,\eta_0)\subset B$. Since $\eta_0=\sup A$, then $A\cap B\neq\emptyset$ but this is an obvious contradiction with the definitions of $A$ and $B$.

$\bullet$ If $f_{\eta_0}(0)>0$ then $f_{\eta_0}'(0)=0$. This comes from the fact that $\eta_0=\sup A\not\in A$ (since $A$ is an open set) and thus $f_{\eta_0}'(0)\geq0$. But we cannot have $f_{\eta_0}'(0)>0$ since the set of interface points $\eta$ for which $f_{\eta}(0)>0$ and $f_{\eta}'(0)>0$, if nonempty, is also an open set.

\medskip

From all these considerations, it follows that either $f_{\eta_0}(0)>0$ but with $f_{\eta_0}'(0)=0$, which leads to a good profile with interface satisfying assumption (P1) in Definition \ref{def1}, or $f_{\eta_0}(0)=0$ with $f_{\eta_0}>0$ on $(0,\eta_0)$ and it is easy to check that this is a good profile with interface satisfying assumption (P2) in Definition \ref{def1}, ending the proof. For a fully-detailed proof the reader is referred to \cite[Section 3]{IS2}.

\medskip

\noindent \textbf{Explicit profile.} One can check by direct calculation that the explicit profile defined in Eq. \eqref{expl.prof} for $\sigma=\sigma_*=\sqrt{2(m+1)}$ is a solution to \eqref{SSODE}. Moreover, in the phase-space variables $X$, $Y$, $Z$, this explicit profile is a straight line of equations
\begin{equation*}
\begin{split}
&Y=-\frac{m-1}{\sigma_*+2}+\frac{\sigma_*+2}{m-1}X,\\
&Z=\frac{m\sigma_*+m+1}{\sigma_*}-\frac{(m\sigma_*+m+1)(\sigma_*+2)}{(m-1)^2}X,
\end{split}
\end{equation*}
connecting the critical point $P_2$ with one of the critical points with interface behavior $P_1^{\gamma}$, for $\gamma=(m\sigma_*+m+1)/\sigma_*$. These easy verifications are left to the reader.
\end{proof}
We represent this explicit line and the orbits going out of $P_0$ (and entering $Q_4$) in the phase space in Figure \ref{fig4}.

\begin{figure}[ht!]
  \begin{center}
  \includegraphics[width=14cm,height=11cm]{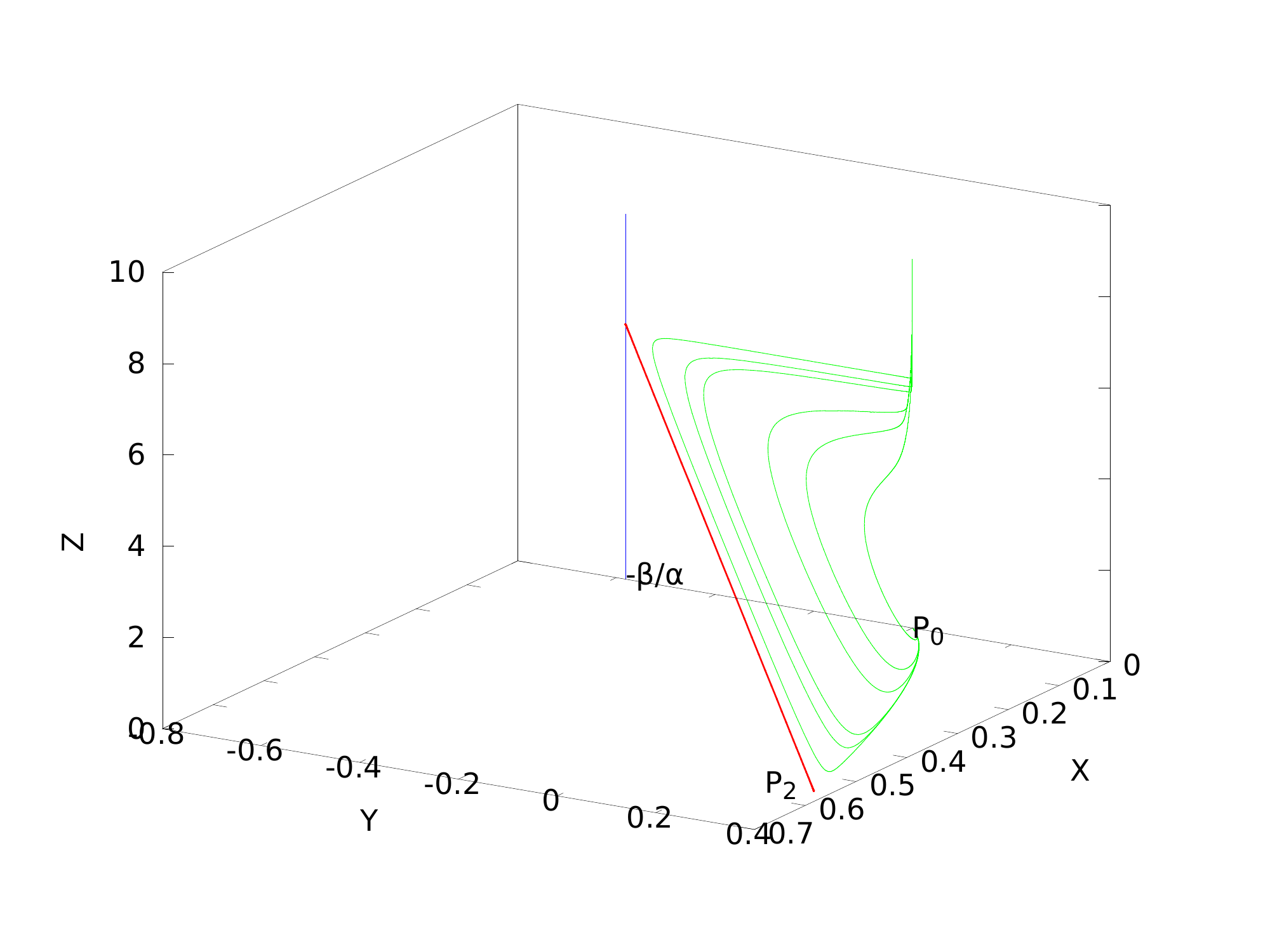}
  \end{center}
  \caption{Orbits going out from the critical points $P_2$ and $P_0$ in the phase space for the critical case $\sigma=\sigma_*$. Experiment for $m=4$.} \label{fig4}
\end{figure}

\section{Blow up profiles for $\sigma>0$ small}\label{sec.small}

This section is devoted to the proof of Theorem \ref{th.small}. As already explained in the Introduction, let us stress here that for exponents $1<p<m$, as it is done in our companion paper \cite[Section 4]{IS2}, the analogous result to Theorem \ref{th.small} was established by a proof based on continuity arguments and on the fact that the homogeneous case $\sigma=0$ can be seen as a limit case when $p>1$. For $p=1$, this technique can no longer work, since for $\sigma=0$ the qualitative behavior is radically different (there is no finite time blow up). Thus, in order to prove Theorem \ref{th.small}, we have to work deeply with the phase space associated to the system \eqref{PSsyst} in variables $X$, $Y$, $Z$. More precisely, the conclusion of Theorem \ref{th.small} follows as an outcome of Theorem \ref{th.exist} and of the following technical result.
\begin{proposition}\label{prop.small}
There exists $\sigma_0>0$ such that for any $\sigma\in(0,\sigma_0)$, all the profiles contained in the orbits going out from the finite critical points $P_0$ and $P_2$ in the phase space connect to the critical point at infinity $Q_4$, that is, they are not compactly supported.
\end{proposition}
\begin{proof}
The proof is rather technical and it involves the construction of several suitable barriers (in the form of planes and, more generally, surfaces in the phase space) for the orbits going out of $P_2$ and $P_0$. We do these constructions starting with $P_2$ as a reference point and for the readers' convenience we divide the proof into several steps.

\medskip

\noindent \textbf{Step 1.} Consider the surface $\{XZ=k_1\}$ in the phase space associated to the system \eqref{PSsyst}, with
$$
k_1=\left(\frac{\beta}{2\alpha}\right)^2=\frac{(m-1)^2}{4(\sigma+2)^2}>0,
$$
for $\sigma\in(0,2)$. The normal to the surface at any of its points has the direction given by the vector $(Z,0,X)$, thus the flow of the system on the surface is given by
$$
XZ[(m-1)Y-2X]+\sigma X^2Z=XZ\left[(m-1)Y+(\sigma-2)X\right],
$$
which is negative for $Y<0$, since $0<\sigma<2$.

\medskip

\noindent \textbf{Step 2.} Consider in a second step the plane $\{Y=-\beta/2\alpha\}$. The flow of the system \eqref{PSsyst} on this plane is given by
\begin{equation*}
\begin{split}
-Y^2-\frac{\beta}{\alpha}Y+X-XY-XZ&=-\frac{\beta^2}{4\alpha^2}+\frac{\beta^2}{2\alpha^2}+X+\frac{\beta}{2\alpha}X-XZ\\
&=\frac{\beta^2}{4\alpha^2}+X\left(1+\frac{\beta}{2\alpha}\right)-XZ,
\end{split}
\end{equation*}
which has a positive sign when
\begin{equation}\label{interm2}
Z<\frac{\beta}{2\alpha}+1+\frac{\beta^2}{4\alpha^2}\frac{1}{X}=\frac{\beta}{2\alpha}+1+\frac{k_1}{X}
\end{equation}
Since it is obvious that $k_1/X$ is smaller than the right hand side in \eqref{interm2}, we obtain that the hyperbola $\{XZ=k_1, Y=-\beta/2\alpha\}$ is situated on the same side of the hyperbola
$$
\left\{Z=\frac{\beta}{2\alpha}+1+\frac{k_1}{X}, \ Y=-\frac{\beta}{2\alpha}\right\},
$$
where the direction of the flow over the plane $\{Y=-\beta/2\alpha\}$ changes.

\medskip

\noindent \textbf{Step 3.} The flow of the system \eqref{PSsyst} on the plane $\{Y=0\}$ is given by $X(1-Z)$. Thus, a connection in the phase space coming from the positive region for $Y$ (that is, having points with $Y>0$) can only cross the plane $\{Y=0\}$ through points with $Z>1$. Since $Z(\xi)=\xi^{\sigma}$ is strictly increasing, once an orbit crossed $\{Y=0\}$ at a point with $Z>1$, then $Z$ will always remain bigger than 1 along that orbit, thus it can never come back to the half-space $\{Y>0\}$ in the future.

\medskip

\noindent \textbf{Step 4. Conclusion of Steps 1-3.} Suppose that we follow an orbit in the phase space coming from the half-space $\{Y>0\}$ (such as, for example, the one going out of the critical point $P_2$ does). If this connection intersects the plane $\{Y=0\}$ at a point in the region lying below the hyperbola $\{XZ=k_1\}$ (more precisely, the intersection point lies in the region $\{XZ<k_1, Y=0\}$), then by Step 1 and Step 3, it will afterwards always remain in the region $\{XZ<k_1, Y<0\}$. Using now Step 2, the possible intersection point with the plane $\{Y=-\beta/2\alpha\}$ would also lie in a region with $XZ<k_1$, hence the orbit we are following cannot cross the plane $\{Y=-\beta/2\alpha\}$. It follows that this orbit remains forever in the region $\{XZ<k_1, -\beta/2\alpha<Y<0\}$. Since $\dot{X}<0$ and $\dot{Z}>0$ in the region $\{XZ<k_1, -\beta/2\alpha<Y<0\}$, it follows that $X$ is decreasing and $Z$ is increasing along the orbit. From the monotonicity of $X$ and $Z$ and the invariance of the $\omega$-limit set of any orbit (see \cite[Theorem 2, Section 3.2]{Pe}), we readily deduce that this orbit has to enter a critical point which necessarily lies in the closure of the region $\{XZ<k_1, -\beta/2\alpha<Y<0\}$. By the local analysis of the critical points, it compulsory has to enter the critical point at infinity denoted by $Q_4$.

\medskip

The aim of the next (and last) steps in the proof is to show that indeed the orbits coming out of the critical points $P_2$ and $P_0$ in the phase space are in the previous situation, that is, they have to cross the plane $\{Y=0\}$ through a point with $XZ<k_1$, at least for $\sigma>0$ sufficiently small.

\medskip

\noindent \textbf{Step 5.} Let us denote the $X$ and $Y$ components of the critical point $P_2$ by
$$
X(P_2):=\frac{m}{2(m+1)\alpha}, \quad Y(P_2):=\frac{1}{(m+1)\alpha}, \quad Z(P_2)=0.
$$
Consider now on the one hand the plane $\{X=X(P_2)\}$. The flow of the system \eqref{PSsyst} on this plane is given by
$$
X[(m-1)Y-2X]=(m-1)X\left[Y-Y(P_2)\right]
$$
which has negative sign in the region $\{Y<Y(P_2)\}$. On the other hand, consider also the plane $\{Y=Y(P_2)\}$. The flow of the system \eqref{PSsyst} on this plane is given by
\begin{equation*}
\begin{split}
-Y^2-\frac{\beta}{\alpha}Y-XY+X-XZ&=-\frac{1}{(m+1)^2\alpha^2}-\frac{\beta}{(m+1)\alpha^2}-\frac{X}{(m+1)\alpha}+X-XZ\\
&<X\left(1-\frac{1}{(m+1)\alpha}\right)-\frac{(m+1)\beta+1}{(m+1)^2\alpha^2}\\
&=\frac{1}{(m+1)^2\alpha^2}\left[\alpha(m+1)(\alpha(m+1)-1)X-\beta(m+1)-1\right],
\end{split}
\end{equation*}
which has negative sign, provided
$$
X<\frac{1+\beta(m+1)}{\alpha(m+1)(\alpha(m+1)-1)}=X(P_2).
$$

\medskip

\noindent \textbf{Step 6.} As the last geometric barrier, consider the plane given by the equation
\begin{equation}\label{interm3}
Y+\frac{Z}{1+\sigma}=1,
\end{equation}
with normal vector $(0,1,1/(1+\sigma))$ and the flow of the system \eqref{PSsyst} on this plane given by
$$
-Y^2-\frac{\beta}{\alpha}Y+X-XZ-X\left(1-\frac{Z}{1+\sigma}\right)+\frac{\sigma}{1+\sigma}XZ=-Y^2-\frac{\beta}{\alpha}Y,
$$
which is negative when $Y>0$.

\medskip

\noindent \textbf{Step 7. End of the proof.} Let us begin from the (unique) connection in the phase space coming out of the point $P_2$. Since it goes out of $P_2$ tangent to the eigenvector $e_3$ given in \eqref{eigen.P2}, it follows that the orbit starting from $P_2$ goes out in the region $\{X<X(P_2), Y<Y(P_2), Z>0\}$. Moreover, we also notice that, exactly at the point $P_2$, we have
$$
Y(P_2)+\frac{Z(P_2)}{1+\sigma}=\frac{1}{(m+1)\alpha}=\frac{(m-1)\sigma}{(m+1)(\sigma+2)}<1,
$$
hence the orbit coming out of $P_2$ also does in inside the region where $Y+Z/(1+\sigma)<1$. Gathering the signs of the flows of the system on the planes considered in Steps 4, 5 and 6, it follows that this orbit cannot cross any of these planes, that is, while $Y>0$, the orbit will stay inside the region
\begin{equation}\label{interm4}
\left\{0<X<X(P_2), \ 0<Y<Y(P_2), \ Y+\frac{Z}{1+\sigma}\leq1\right\}.
\end{equation}
Since this orbit has to cross the plane $\{Y=0\}$, according to \eqref{interm4} it crosses it at a point with $X<X(P_2)$ and $Z<1+\sigma$, that is,
$$
XZ<\frac{(1+\sigma)(m-1)^2\sigma}{2(m+1)(\sigma+2)}<k_1,
$$
provided $\sigma\in(0,\sigma_0)$ for some $\sigma_0$ sufficiently small (more precisely, $\sigma_0$ is the unique solution to the third degree algebraic equation $x(x+1)(x+2)=m+1$). Joining this with Step 4, we deduce that the orbit coming out of $P_2$ should connect to the critical point at infinity $Q_4$ and thus have noncompact support, for any $\sigma\in(0,\min\{\sigma_0,2\})$ (recall that the restriction $\sigma<2$ was needed at Step 1).

Finally, the connections coming out from the critical point $P_0$ in the phase space, also enter immediately the same region given in \eqref{interm4}, as they contain profiles starting with positive slope at $\xi=0$ (that is, $Y>0$ in a right neighborhood of the origin). Then, the above is also valid for all the orbits coming out of $P_0$.
\end{proof}
We show in Figure \ref{fig5} below the evolution of the orbits going out from $P_2$ and $P_0$ and entering $Q_4$, as proved above.

\begin{figure}[ht!]
  \begin{center}
  \includegraphics[width=14cm,height=11cm]{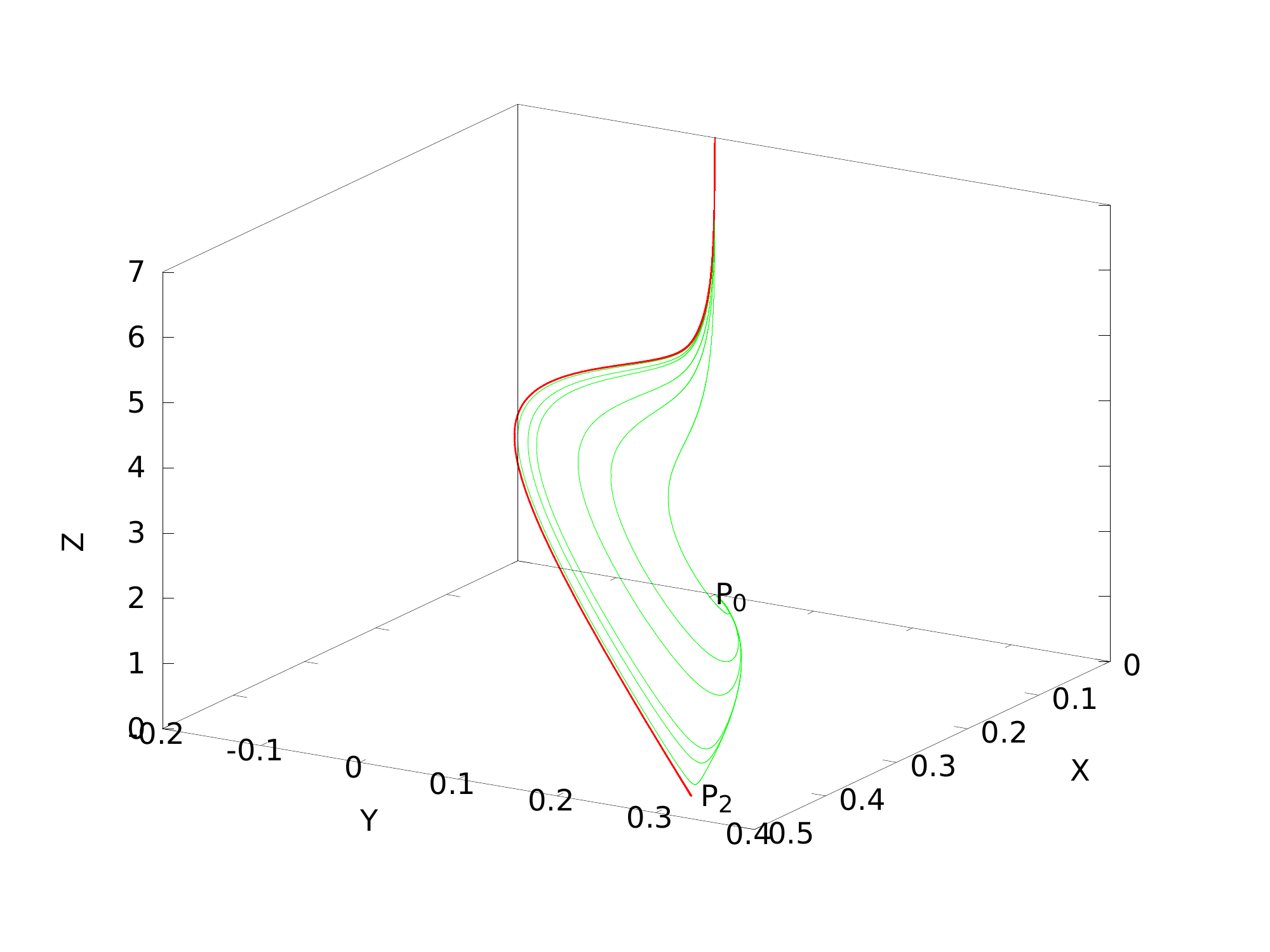}
  \end{center}
  \caption{Orbits going out from the critical points $P_2$ and $P_0$ in the phase space. Experiment for $m=4$ and $\sigma=2$.} \label{fig5}
\end{figure}

Since by Theorem \ref{th.exist}, for any $\sigma>0$ there exists a good profile with interface in the sense of Definition \ref{def1}, and by Proposition \ref{prop.small} above, such good profile with interface cannot be contained in an orbit coming either from $P_0$ or from $P_2$ in the phase space for $0<\sigma<\min\{\sigma_0,2\}$, it necessarily follows that any good profile with interface (which may not be unique) for $\sigma\in(0,\min\{\sigma_0,2\})$ comes from the critical point at infinity $Q_1$ and thus begins with $f(0)=A$, $f'(0)=0$ for some $A>0$, concluding the proof of Theorem \ref{th.small}.

\section{Blow up profiles for $\sigma>0$ large}\label{sec.large}

In this section, we analyze the phase space associated to the system \eqref{PSsyst} for $\sigma>0$ sufficiently large, with the aim of completing the proof of Theorem \ref{th.large}. We start with an easy but important remark.
\begin{lemma}\label{lem.large1}
Let $\sigma>0$ fixed. If an orbit in the phase space crosses the plane $\{Y=-\beta/\alpha\}$, then it cannot enter again in the half-space $\{Y>-\beta/\alpha\}$.
\end{lemma}
\begin{proof}
Consider the plane $\{Y=-\beta/\alpha\}$. The flow of the system \eqref{PSsyst} on this plane is given by
$$
-Y^2-\frac{\beta}{\alpha}Y+X-XY-XZ=X\left(1+\frac{\beta}{\alpha}-Z\right),
$$
hence an orbit in the phase space coming from the region $\{Y>-\beta/\alpha\}$ can cross this plane only in the region where the above flow is negative, that is, at points where $Z>1+\beta/\alpha$. Since $Z$ is nondecreasing, it follows that after the crossing point, the inequality $Z>1+\beta/\alpha$ will be satisfied forever on the orbit we are dealing with, thus the orbit cannot cross again the plane $\{Y=-\beta/\alpha\}$.
\end{proof}
We are now in a position to establish the behavior of the profiles contained in the unique connection in the phase space starting from the critical point $P_2$. This is the main technical ingredient in the proof of Theorem \ref{th.large}.
\begin{lemma}\label{lem.large2}
There exists $\sigma_1>0$ sufficiently large such that, for any $\sigma\in(\sigma_1,\infty)$, the orbit coming out of the critical point $P_2$ in the phase space enters the critical point at infinity denoted by $Q_3$ on the Poincar\'e hypersphere.
\end{lemma}
\begin{proof}
Let us first recall that, for $\sigma=\sigma_*=\sqrt{2(m+1)}$, there exists an explicit profile with interface \eqref{expl.prof} contained in an orbit in the phase space coming out of $P_2$ and entering a critical point belonging to the plane $\{Y=-\beta_*/\alpha_*\}$. For $\sigma>0$, denote by $l_{\sigma}$ the unique orbit coming out of $P_2$ in the phase space. Let now $y_0\in (-\beta_*/\alpha_*,0)$ be fixed and set
\begin{equation}\label{interm5}
\sigma_{y_0}:=\sup\{\sigma>0: l_{\sigma} \ {\rm crosses \ the \ plane } \ Y=y_0\}.
\end{equation}
Since $y_0\in (-\beta_*/\alpha_*,0)$, it is obvious that $\sigma_{y_0}>\sigma_*$. Our next goal is to prove that $\sigma_{y_0}=+\infty$, at least for $y_0$ sufficiently close to zero. Assume then by contradiction that $\sigma_{y_0}<\infty$. By continuity, the connection denoted by $l_{\sigma_{y_0}}$ should then intersect the plane $\{Y=y_0\}$ either finishing to a critical point inside the plane, or being tangent to it.

Assume that the orbit $l_{\sigma_{y_0}}$ is tangent to the plane $\{Y=y_0\}$. The flow on the plane $\{Y=y_0\}$ is given by the sign of the following expression
\begin{equation}\label{flow}
F(X,Z;y_0):=-y_0^2-\frac{\beta}{\alpha}y_0+X-Xy_0-XZ,
\end{equation}
hence there exists a hyperbola inside this plane where the flow of the system on the plane changes, given by the equation:
\begin{equation}\label{interm6}
Z=\frac{1}{X}\left[X-Xy_0-y_0^2-\frac{\beta}{\alpha}y_0\right].
\end{equation}
By the definition of $\sigma_0$, for any $\sigma\in(0,\sigma_0)$ the orbit $l_{\sigma}$ crosses the plane $\{Y=y_0\}$, and should do that necessarily by a point in the region where $F(X,Z;y_0)<0$. On the other hand, for any $\sigma>\sigma_0$ the orbit $l_{\sigma}$ do not intersect the plane $\{Y=y_0\}$. Thus, since $\sigma_0<\infty$, it is easy to infer by a continuity argument that the tangency point of $l_{\sigma_{y_0}}$ with the plane $\{Y=y_0\}$ lies exactly on the hyperbola defined in \eqref{interm6}. This means in particular that $Y'=0$ at this point, and we can calculate
\begin{equation}\label{interm7}
\begin{split}
Y''&=\left[-2Y-\frac{\beta}{\alpha}-X\right]Y'+X'-X'Y-X'Z-XZ'\\
&=X[(m-1)y_0-2X](1-y_0)-XZ[(m-1)y_0-2X]-\sigma X^2Z.
\end{split}
\end{equation}
Replacing $Z$ from the hyperbola \eqref{interm6} into \eqref{interm7}, we obtain after direct (although rather long and tedious) calculations that
\begin{equation}\label{interm8}
\begin{split}
Y''&=\sigma_0X^2(y_0-1)+(m-1)y_0^3\\&+\frac{1}{\sigma_0+2}\left[(\sigma_0-2)Xy_0(m-1+(\sigma_0+2)y_0)+m(m-2)y_0^2+y_0\right].
\end{split}
\end{equation}
From Lemma \ref{lem.large1} we know that, since the orbit $l_{\sigma_{y_0}}$ touches the plane $\{Y=y_0\}$, we have
$$
-\frac{\beta(\sigma_0)}{\alpha(\sigma_0)}=-\frac{m-1}{\sigma_0+2}\leq y_0<0,
$$
whence, also recalling that $\sigma_0>\sigma_*=\sqrt{2(m+1)}>2$, we get that
$$
(\sigma_0-2)Xy_0(m-1+(\sigma_0+2)y_0)<0.
$$
Thus, choosing $y_0$ such that
\begin{equation}\label{interm9}
\left\{\begin{array}{ll}-\frac{\beta_*}{\alpha_*}<y_0<0, & {\rm if} \ 1<m\leq2,\\
-\frac{1}{m(m-2)}<y_0<0, & {\rm if} \ m>2,\end{array}\right.
\end{equation}
we conclude that all the terms in \eqref{interm8} are non-positive and some of them is strictly negative, so that $Y''<0$ at the tangency point between the orbit $l_{\sigma_{y_0}}$ and the plane $\{Y=y_0\}$. But this is a contradiction with the fact that the orbit touches the plane $\{Y=y_0\}$ coming from the region $\{Y>y_0\}$: indeed, in such case, at the tangency point we would have a local minimum with respect to the variable $Y$ on the orbit $l_{\sigma_{y_0}}$, that is, $Y''\geq0$. This contradiction shows that the orbit $l_{\sigma_{y_0}}$ cannot be tangent to the plane $\{Y=y_0\}$.

It remains the case where the connection $l_{\sigma_{y_0}}$ ends in a finite critical point inside the plane $\{Y=y_0\}$. But this means that necessarily
$y_0=-(m-1)/(\sigma_0+2)$ and $l_{\sigma_{y_0}}$ connects thus the critical point $P_2$ with one of the critical points $P_1^{\gamma}$, thus containing a good profile with interface. We already know that this is possible (for example, for the explicit case $\sigma=\sigma_*$), but on the other hand, it is shown by a transversality argument in \cite[Subsection 5.3 and 5.4]{IS2} that this can happen only for a discrete set of parameters $\sigma>0$. Thus, for any $y_0$ satisfying condition \eqref{interm9} except for at most a discrete set, we find that $\sigma_{y_0}=\infty$. By the definition of $\sigma_0$ as a supremum in \eqref{interm5}, it follows that for
$$
\sigma=-2-(m-1)/y_0, \quad {\rm that \ is} \ y_0=-\frac{\beta(\sigma)}{\alpha(\sigma)},
$$
the connection $l_{\sigma}$ coming out of $P_2$ crosses the plane $\{Y=y_0\}$. According to Lemma \ref{lem.large1}, $l_{\sigma}$ will remain forever in the half-space $\{Y<y_0\}$. Writing the equation for $\dot{Y}$ in the form 
$$
\dot{Y}=-Y\left(Y+\frac{\beta}{\alpha}+X\right)+X(1-Z),
$$
taking into account the fact that $Z$ is increasing, $X$ is decreasing along the orbit and that $Z>1$ was already achieved when crossing the plane $\{Y=0\}$, it readily follows that, if for some $\eta_0>0$ we have $\dot{Y}(\eta_0)<0$, then $Y(\eta)<0$ for any $\eta>\eta_0$. Since at the moment of crossing the plane $\{Y=-\beta/\alpha\}$ it is obvious that $\dot{Y}<0$, we infer that $Y$ is decreasing along the orbit in the region $\{Y<-\beta/\alpha\}$. Thus, this orbit has to enter a critical point, thus it eventually enters the only critical point with $Y<-\beta/\alpha$ which is the attractor $Q_3$.

Finally, it is obvious from the proof of Theorem \ref{th.exist} that the exceptional discrete set where we have a good profile with interface contained in an orbit coming out of the point $P_2$ should be bounded (and thus finite), concluding the proof.
\end{proof}
We still need one more technical result in the phase space before completing the proof of Theorem \ref{th.large}.
\begin{lemma}\label{lem.large3}
For any $\sigma>0$, there exists an orbit in the phase space connecting the critical points $P_2$ and $P_0$ which in included in the invariant plane $\{Z=0\}$.
\end{lemma}
\begin{proof}
We restrict ourselves to the invariant plane $\{Z=0\}$, thus reducing the system \eqref{PSsyst} to the following system
\begin{equation}\label{PPsystZ0}
\left\{\begin{array}{ll}\dot{X}=X[(m-1)Y-2X],\\
\dot{Y}=-Y^2-\frac{\beta}{\alpha}Y+X-XY.\end{array}\right.
\end{equation}
Notice that the two critical points $P_2$ and $P_0$ both lie in the plane $\{Z=0\}$. We begin with two particular curves in the phase plane associated to the system \eqref{PPsystZ0}. First of all, let us consider the line
\begin{equation}\label{curve1}
(m-1)Y-2X=0, \quad {\rm that \ is} \ Y=\frac{2}{m-1}X,
\end{equation}
which connects the two critical points $P_2$ and $P_0$. The flow of the system \eqref{PPsystZ0} over the curve \eqref{curve1} is given by the sign of the following expression
\begin{equation*}
\begin{split}
(m-1)\left[-Y^2-\frac{\beta}{\alpha}Y+X-XY\right]&=\frac{2(m+1)}{m-1}X\left[-X+\frac{\sigma(m-1)^2}{2(m+1)(\sigma+2)}\right]\\
&=\frac{2(m+1)}{m-1}X[X(P_2)-X],
\end{split}
\end{equation*}
which is positive in the region $X<X(P_2)$. We now consider the curve where $dY/dX=0$, of equation
\begin{equation}\label{curve2}
-Y^2-\frac{\beta}{\alpha}Y+X-XY=0,
\end{equation}
whose normal vector has the direction $(1-Y,-2Y-\beta/\alpha-X)$ and the flow of the system \eqref{PPsystZ0} over the curve \eqref{curve1} is given by the sign of the following quantity:
$$
X(1-Y)[(m-1)Y-2X],
$$
which in the region $Y<Y(P_2)=1/(m+1)\alpha<1$ it is easy to check that it is positive. Noticing that the curve \eqref{curve2} also connects the critical points $P_0$ and $P_2$, it follows from the previous analysis that an orbit of the phase plane associated to the system \eqref{PPsystZ0} can only enter the closed region limited by the points $P_0$, $P_2$ and the curves \eqref{curve1} and \eqref{curve2} from outside, but never go out of this closed region once inside it.

We readily notice from Lemma \ref{lem.1} and its proof that the connections going out of $P_0$ tangent to its center manifold (which is the restriction of the two-dimensional center manifold near $P_0$ to the invariant plane $\{Z=0\}$) have the slope
$$
\frac{dY}{dX}=\frac{-Y^2-\beta/\alpha Y+X-XY}{X[(m-1)Y-2X]}=\frac{\alpha}{\beta}>0.
$$
But analyzing the sign of the fraction in the right hand side of the first equality above, it follows that near $P_0$ it is positive only in the closed region limited by the points $P_0$ and $P_2$ and the curves \eqref{curve1} and \eqref{curve2}. This means that the orbits going out of $P_0$ tangent to its center manifold should go into this closed region and thus remain forever there. They then have to enter the attractor (for the system \eqref{PPsystZ0}) $P_2$.
\end{proof}
We can see the connection from $P_0$ to $P_2$ in Figure \ref{fig6}. In the same Figure \ref{fig6} one can see a visual representation of the steps of the proof of Lemma \ref{lem.large3} with the curves \eqref{curve1} and \eqref{curve2} limiting regions with different behavior in the phase plane associated to the system \eqref{PPsystZ0}.

\begin{figure}[ht!]
  \begin{center}
  \includegraphics[width=12cm,height=9cm]{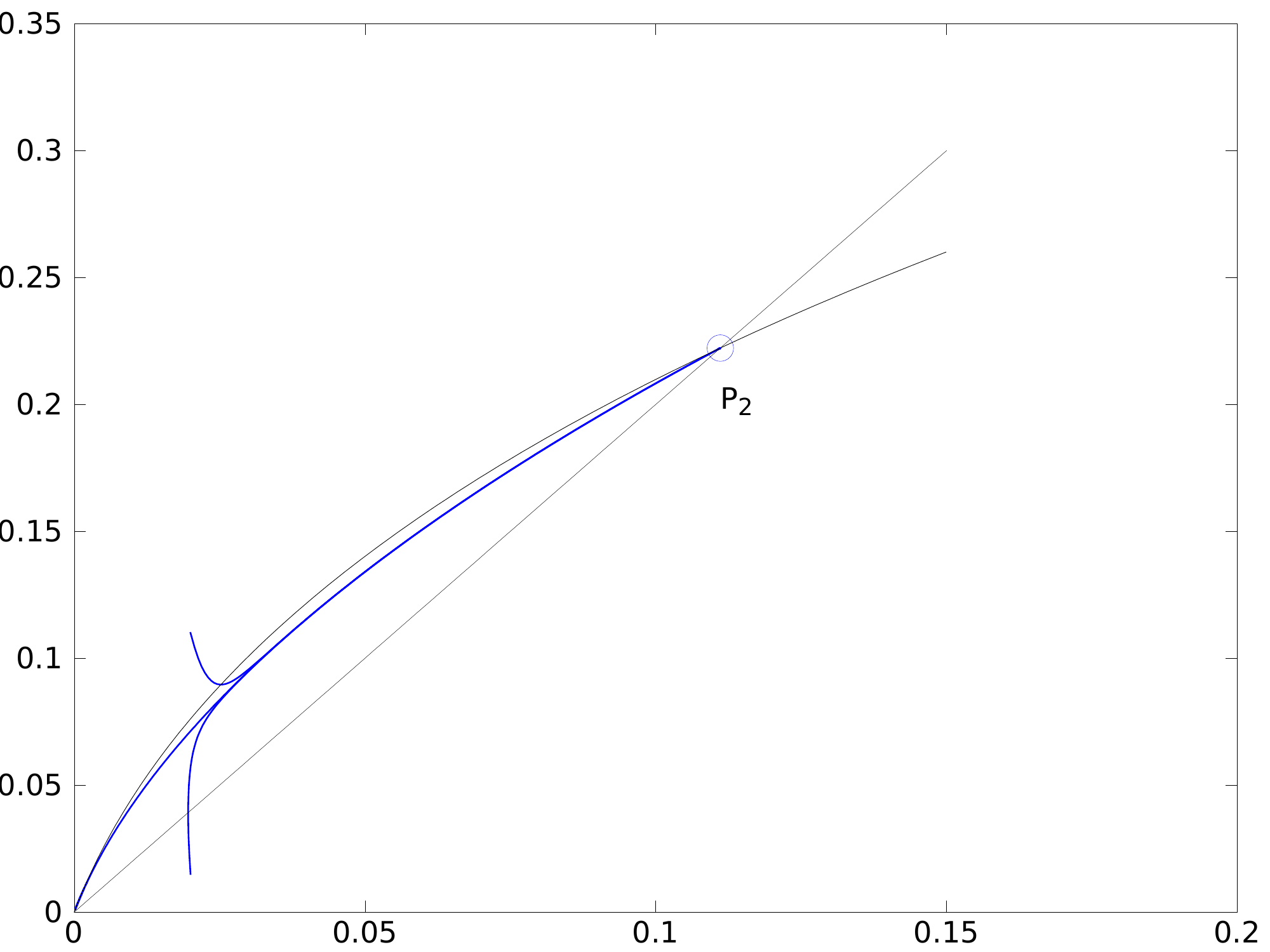}
  \end{center}
  \caption{Orbit connecting $P_0$ and $P_2$ inside the invariant plan $\{Z=0\}$.} \label{fig6}
\end{figure}

We need one more preliminary result concerning the orbits going out of the point $P_0$.
\begin{lemma}\label{lem.large4}
Let $\sigma>0$ be such that in the phase space associated to the system \eqref{PSsyst} there are orbits going out of $P_0$ and entering $Q_4$ and at the same time orbits going out of $P_0$ and entering $Q_3$. Then there exists at least one good profile with interface contained in an orbit going out of $P_0$.
\end{lemma}
\begin{proof}
We know by Lemma \ref{lem.1} that the profiles contained in orbits going out of $P_0$ are uniquely characterized by the constant $k>0$ such that $Z\sim kX$ near $\xi=0$. Moreover, the local analysis in Section \ref{sec.local} shows that an orbit coming out of $P_0$ can either connect to $Q_4$, to $Q_3$ or to any of the interface points $P_1^{\gamma}$ with $\gamma>0$. We can thus define the following sets:
\begin{equation*}
\begin{split}
&A_0:=\{k>0: {\rm the \ orbit \ from \ }P_0 \ {\rm with} \ Z\sim kX \ {\rm enter} \ Q_4\},\\
&B_0:=\{k>0: {\rm the \ orbit \ from \ }P_0 \ {\rm with} \ Z\sim kX \ {\rm enter} \ P_1^{\gamma} \ {\rm for \ some} \ \gamma>0\},\\
&C_0:=\{k>0: {\rm the \ orbit \ from \ }P_0 \ {\rm with} \ Z\sim kX \ {\rm enter} \ Q_3\}.
\end{split}
\end{equation*}
Since $Q_3$ is a stable node and $Q_4$ behaves like an attractor for the orbits entering it from the finite region of the phase space (see Lemma \ref{lem.Q42} and the end of its proof), it readily follows that $A_0$ and $C_0$ are open sets, while the hypothesis implies that $A_0$, $C_0$ are non-empty. Thus, $B_0$ is a non-empty closed set.
\end{proof}
\noindent \textbf{Remark.} The connection from $P_2$ to $P_0$ inside the invariant plane $\{Z=0\}$ established in Lemma \ref{lem.large3} corresponds to the limit case in Lemma \ref{lem.1} such that $Z\sim kX$ with $k=0$. 

Theorem \ref{th.tail} is now a simple and immediate consequence of the previous proofs
\begin{proof}[Proof of Theorem \ref{th.tail}]
The existence of $\sigma_0$ such that for any $\sigma\in(0,\sigma_0)$ all the orbits coming out of both $P_0$ and $P_2$ in the phase space enter $Q_4$ is the outcome of Proposition \ref{prop.small}. But passing to profiles, such connections contain all the profiles with $f(0)=0$ and $(f^m)'(0)=0$, while entering $Q_4$ means in terms of profiles behaving as in \eqref{tail} as $\xi\to\infty$, see Lemma \ref{lem.Q42}. On the other hand, the local analysis done in Lemma \ref{lem.Q42} also shows that for any $\sigma>0$ there exist cycles inside the elliptic sector near the origin of the system \eqref{PSsyst2}, that is, connections from $P_0$ to $Q_4$ in the phase space associated to the system \eqref{PSsyst}. The profiles contained in these orbits behave as desired both as $\xi\to0$ and as $\xi\to\infty$.
\end{proof}

We are now ready to complete the proof of Theorem \ref{th.large}
\begin{proof}[Proof of Theorem \ref{th.large}]
Let $\sigma>\sigma_1$ be fixed. By Lemma \ref{lem.large2} we know that the orbit coming out of $P_2$ in the phase space connects to the stable node $Q_3$ at infinity. Since $Q_3$ is a stable node and $P_2$ a saddle point, there exists $\delta>0$ sufficiently small such that for any (non-critical) point in a small half-ball near $P_2$, namely $(X,Y,Z)\in B(P_2,\delta)\cap\{Z>0\}$, the unique orbit passing through this point in the phase-space enters $Q_3$. On the other hand, by Lemma \ref{lem.large3} we know that there is a orbit connecting $P_0$ to $P_2$ inside the plane $\{Z=0\}$. Again by continuity, as the two points are in the finite region, there exists an orbit going out of $P_0$ and approaching as much as we want $P_2$ (without entering this point), in particular, there exists such an orbit entering the ball $B(P_2,\delta)$. Thus, we conclude that for any $\sigma>\sigma_1$ (with $\sigma_1$ given by Lemma \ref{lem.large2}), there exists at least an orbit in the phase space coming out of $P_0$ and entering $Q_3$. On the other hand, we know from Proposition \ref{prop.small} that there exists $\sigma_0\in(0,2)$ such that for $\sigma\in(0,\sigma_0)$, all the connections coming out of $P_0$ in the phase space enter the critical point $Q_4$. Define the following three sets:
\begin{equation*}
\begin{split}
&A:=\{\sigma>0: {\rm all \ the \ orbits \ from \ }P_0 \ {\rm enter} \ Q_4\},\\
&B:=\{\sigma>0: {\rm there \ are \ orbits \ from \ }P_0 \ {\rm with \ different \ behavior}\},\\
&C:=\{\sigma>0: {\rm all \ the \ orbits \ from \ }P_0 \ {\rm enter} \ Q_3\}.
\end{split}
\end{equation*}
Since $Q_3$ is a stable node, it is easy to check that $A$ is an open set. Moreover, Proposition \ref{prop.small} shows that $A$ is nonempty and contains an interval $(0,\sigma_0)$, while Lemmas \ref{lem.large2} and \ref{lem.large3} show that $A\neq(0,\infty)$. On the other hand, it follows from Theorem \ref{th.tail} that $C$ is the empty set, as there is for any $\sigma>0$ at least a connection between $P_0$ and $Q_4$. It thus follows that $B$ is a non-empty closed set. Let now $\sigma\in B$. Then either there are orbits from $P_0$ entering one of the points $P_1^{\gamma}$ for some $\gamma>0$, or there are some orbits from $P_0$ entering $Q_4$ and other orbits from $P_0$ entering $Q_3$; but in the latter case, Lemma \ref{lem.large4} insures that there exists a good profile with interface belonging to an orbit going out of $P_0$. Thus, for any $\sigma\in B$ there exists a good profile with interface (in the sense of Definition \ref{def1}) coming from the critical point $P_0$. Moreover, the latter situation occurs for any $\sigma\in(\sigma_1,\infty)$ with $\sigma_1>0$ given by Lemma \ref{lem.large2}, hence for any $\sigma\in(\sigma_1,\infty)$ there exists a good profile with interface behaving as in \eqref{buinf}.
\end{proof}
The following numerical experiment illustrates the behavior of the orbits going out of $P_2$ and $P_0$ in the phase space for $\sigma>0$ sufficiently large. We can visualize in Figure \ref{fig7} the splitting of the orbits going out of $P_0$, with respect to their behavior, into the three sets $A_0$, $B_0$ and $C_0$ defined in the proof of Lemma \ref{lem.large4}.

\begin{figure}[ht!]
  \begin{center}
  \includegraphics[width=14cm,height=11cm]{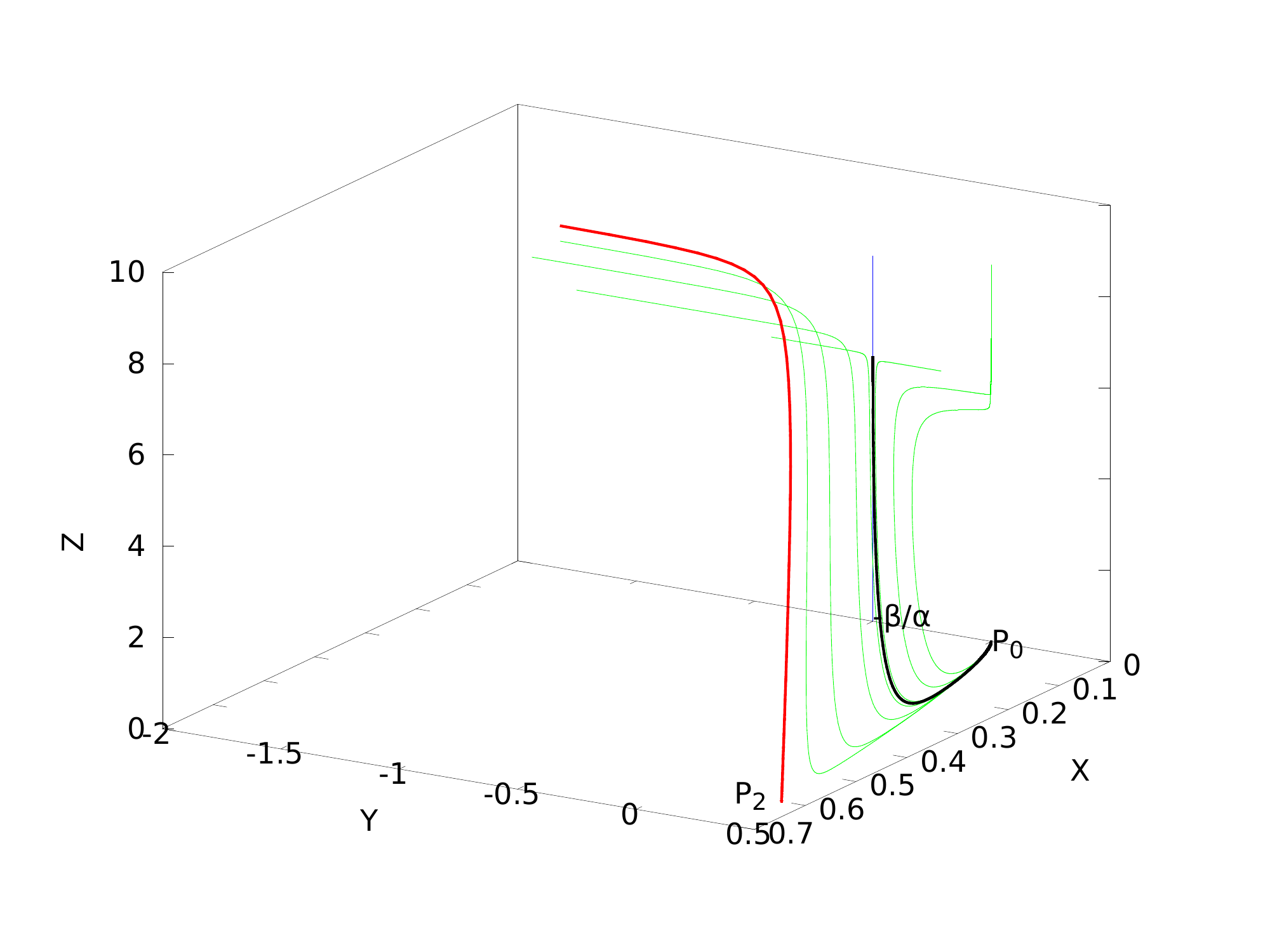}
  \end{center}
  \caption{Orbits going out from the critical points $P_2$ and $P_0$ in the phase space for $\sigma$ large. Experiment for $m=4$ and $\sigma=4$.} \label{fig7}
\end{figure}

\section*{Extensions and open problems}

We gather in this section some question, that in our opinion might be interesting, which can be addressed in some further work in order to complete the panorama of the blow up profiles for \eqref{eq1}.

\medskip

\noindent \textbf{1. Uniqueness results.} Probably the most interesting open question raised by our previous study is whether the good profile with interface for a given $\sigma>0$ is unique. Even in a weaker sense, showing that $\sigma_*=\sqrt{2(m+1)}$ is the only exponent $\sigma>0$ for which there exists a good profile with interface and with behavior as in \eqref{buzero} as $\xi\to0$ would be very interesting. Indeed, if such uniqueness result is true, then we have a full classification of the good profiles with interface: starting with $f(0)=A>0$ and $f'(0)=0$ for $\sigma\in(0,\sigma_*)$, the explicit solution $f_*$ for $\sigma=\sigma_*$ and the profiles given by Theorem \ref{th.large} for $\sigma\in(\sigma_*,\infty)$. We conjecture that this uniqueness is true, and even more, that for any $\sigma>0$ there exists a unique good profile with interface. We performed numerical experiments that support this conjecture, but still a rigorous proof is missing. Using the transversality technique in \cite[Subsection 5.4]{IS2} (which applies in the whole range $1\leq p<m$), we infer that there might be at most a finite set of exponents $\sigma>0$ for which a good profile with interface similar to $f_*$, that is with
$$
f(0)=0, \ \ f(\xi)\sim\left[\frac{m-1}{2m(m+1)}\right]^{1/(m-1)}\xi^{2/(m-1)}, \quad {\rm as} \ \xi\to0
$$
may exist. Indeed, the transversality gives that the set of such exponents $\sigma>0$ is discrete and gathering the results of Propositions \ref{prop.far} and \ref{prop.close} we further get that it is also bounded, thus finite. Still, the step of passing from a finite set to a singleton is missing by now and we leave it as an open problem, conjecturing that the uniqueness holds true.

\medskip

\noindent \textbf{2. Extension to higher space dimensions.} This is a next step that will be addressed in future work. The difficulty of passing to higher space dimension $N\geq2$ comes with the non-autonomous extra-term $(N-1)(f^m)'(\xi)/\xi$ in the equation of the profiles \eqref{SSODE}. Thus, the phase space may differ and new critical exponents are expected to appear. Just as a precedent, it was seen in \cite{S4, GV97} that passing to higher space dimension might be more difficult even with homogeneous reaction (that is, $\sigma=0$), thus for example the mere existence of a "good profile" is an open problem for the well-studied equation
$$
u_t=\Delta u^m+u^p, \quad 1<m<p
$$
for very large $p$ (see \cite[Theorems 4, 5, p. 197]{S4} where existence is only shown for $m<p<m(N+2)/(N-2)$ for $N\geq3$; this range has been extended a bit in \cite{GV97} but not for any $p\in(m,\infty)$ and up to our knowledge the existence of a "good profile" is still an open question for $p$ very large and $N\geq3$). However, the present work shows that even in dimension $N=1$, the influence of the weight $|x|^{\sigma}$ on the blow up profiles and behavior is very interesting and significant.

\section*{Acknowledgements} R. I. is supported by the ERC Starting Grant GEOFLUIDS 633152. A. S. is partially supported by the Spanish project MTM2017-87596-P.

\bibliographystyle{plain}

\end{document}